\documentclass[11pt,reqno]{amsart}

\usepackage{amsmath,amssymb,mathrsfs}
\usepackage{graphicx,cite}
\usepackage{dsfont}
\usepackage{color}
\usepackage{multirow}%
\usepackage{amsmath,amssymb,amsfonts}%
\usepackage{amsthm}%
\usepackage[title]{appendix}%
\usepackage{xcolor}%
\usepackage{textcomp}%
\usepackage{manyfoot}%
\usepackage{booktabs}%
\usepackage{algorithm}%
\usepackage{algorithmicx}%
\usepackage{algpseudocode}%
\usepackage{listings}%
\usepackage{tikz}
\usetikzlibrary{shapes,arrows}
\usepackage{graphicx}
\usepackage{float}
\usepackage{stmaryrd}
\usepackage{bbm}
\usepackage[utf8]{inputenc}
\allowdisplaybreaks[4]
\newtheorem{theorem}{Theorem}[section]%
\newtheorem{assumption}[theorem]{Assumption}%
\newtheorem{remark}[theorem]{Remark}

\newtheorem{lemma}[theorem]{Lemma}
\numberwithin{equation}{section}

\begin{document}

\title[An inverse random source problem]{An inverse random source problem for the fractional Helmholtz equation}

\author{Peijun Li}
\address{SKLMS, ICMSEC, Academy of Mathematics and Systems Science, Chinese Academy of Sciences, Beijing 100190, China}
\email{lipeijun@lsec.cc.ac.cn}

\author{Zhenqian Li}
\address{Academy of Mathematics and Systems Science, Chinese Academy of Sciences, Beijing 100190, China}
\email{lizhenqian2020@amss.ac.cn}

\thanks{The work is supported by the National Key R\&D Program of China (2024YFA1012300).}



\begin{abstract}
This paper investigates an inverse random source problem for the stochastic fractional Helmholtz equation. The source is modeled as a centered, complex-valued, microlocally isotropic generalized Gaussian random field whose covariance and relation operators are described by classical pseudo-differential operators. For sufficiently large wavenumbers, we first establish the well-posedness of the direct problem in the distributional sense by analyzing the corresponding Lippmann--Schwinger integral equation. For the inverse problem, we show that the principal symbols of both the covariance and relation operators can be uniquely determined, with probability one, from the far-field patterns generated by a single realization of the random source. The approach employs a combination of the Born linearization, asymptotic expansions of the fractional Helmholtz Green kernel at high wavenumbers, and microlocal analysis of associated Fourier integral operators.
\end{abstract}

\maketitle

\section{Introduction}

The fractional Helmholtz equation provides a flexible framework for modeling wave propagation in media where classical, local constitutive laws fail. Its fractional Laplacian accounts for intrinsic nonlocality, power-law dispersion, and anomalous attenuation that arise in a variety of physical systems, including quantum optics \cite{Kraisler-Schotland-JMP-23, Kraisler-Schotland-JMP-22}, geophysical electromagnetics \cite{Glusa-Antil-DElia-Waanders-Weiss-SISC-21}, and electrodynamics \cite{Tarasov-Trujillo-AOP-13}. In quantum optical settings, for instance, fractional Helmholtz models can describe photon fields coupled to atomic degrees of freedom through real-space quantization and capture the global, nonlocal characteristics of photon-atom interactions \cite{Hiltunen-Kraisler-Schotland-Weinstein-SIMA-24, Hoskins-Kaye-Rachh-Schotland-JCP-23}. In geophysical electromagnetics, recent studies demonstrate that fractional Helmholtz equations effectively reproduce anomalous magnetotelluric responses and model long-range spatial correlations in self-similar geological formations \cite{Weiss-Waanders-Antil-GJI-20}.

In many applications where fractional Helmholtz models arise, the underlying medium or source exhibits substantial randomness due to complex environmental conditions, measurement noise, or intrinsic stochastic mechanisms governing the physical process. These effects introduce uncertainties that cannot be described by deterministic models, making stochastic formulations both natural and necessary. Incorporating random parameters, however, significantly increases the analytical complexity: the direct problem must be treated with coefficients or sources of low regularity, and inverse problems are recast from recovering deterministic quantities to identifying statistical properties, such as covariance structures or principal symbols of the underlying random fields.

In this paper, we study the stochastic fractional Helmholtz equation
\begin{align}\label{119}
    ( - \Delta )^{ \alpha } u - k^{ 2 \alpha } u + q u = f \quad \text{in} \ \mathbb{ R }^{ d },
\end{align}
where $\alpha \in ( 0, 1 )$ is the fractional order, $d \in \{ 1, 2, 3 \}$ denotes the spatial dimension, and $k > 0$ is the wavenumber. The potential $q$ is a deterministic function, while the source term $f$ is modeled as a complex-valued generalized microlocally isotropic Gaussian (GMIG) random field. The precise regularity assumptions on $q$ and $f$ will be specified later. The fractional Laplacian $( - \Delta )^{ \alpha }$ is defined via the Fourier transform by
\begin{align*}
    ( - \Delta )^{ \alpha } u ( x ) = \frac{ 1 }{ ( 2 \pi )^{ d } } \int_{ \mathbb{ R }^{ d } } e^{ \mathrm{ i } x \cdot \xi } | \xi |^{ 2 \alpha } \hat{ u } ( \xi ) d \xi.
\end{align*}

To ensure well-posedness in the whole space $\mathbb{R}^{d}$, the solution $u$ satisfies the Sommerfeld radiation condition
\begin{align}\label{102}
    \lim \limits_{ r \rightarrow \infty } r^{ \frac{ d - 1 }{ 2 } } ( \partial_{ r } u - \mathrm{i} k u ) = 0, \quad r = | x |.
\end{align}
This condition implies the following asymptotic expansion of the wave field:
\begin{align}\label{ffp}
    u ( x, k ) = \frac{ e^{ \mathrm{i} k |x| }}{ | x |^{ \frac{ d - 1 }{ 2 } } } u^{ \infty } ( \hat{x}, k ) + O ( | x |^{ -N_d }) , \quad  \ | x | \rightarrow \infty,
\end{align}
where $u^{\infty}$ denotes the far-field pattern, $\hat{x}=x/|x|$ is the observation direction, and the decay exponents $N_d$ are given by
\begin{align}\label{Nd}
     N_{ d } =
    \begin{cases}
        1 + 2 \alpha, & d = 1,\\[3pt] 
        \frac{ 3 }{ 2 }, & d = 2,\\[3pt] 
        3 + 2 \alpha, & d = 3.
    \end{cases}
\end{align}

We examine the stochastic fractional Helmholtz model \eqref{119}--\eqref{102} from both direct and inverse perspectives. For the direct problem, our goal is to establish well-posedness, i.e., the existence, uniqueness, and regularity of the solution $u$. For the inverse problem, we aim to extract statistical properties of the underlying random source $f$ from observed far-field data.

On the side of direct problems, considerable progress has been made on the analysis and numerical simulation of equations involving the fractional Laplacian, and in particular the fractional Helmholtz equation. Fast and scalable solvers have been developed in \cite{Glusa-Antil-DElia-Waanders-Weiss-SISC-21, Belevtsov-Lukashchuk-AMC-22} to handle the nonlocal structure and oscillatory nature of fractional Helmholtz operators, motivated in part by applications in wave propagation and electromagnetic modeling. In quantum optics, nonlocal partial differential equations of fractional Helmholtz type have been employed to describe bound states, resonances, and related spectral phenomena \cite{Hiltunen-Kraisler-Schotland-Weinstein-SIMA-24}, while high-order numerical methods have been proposed for simulating single-excitation states \cite{Hoskins-Kaye-Rachh-Schotland-JCP-23}. Kinetic descriptions for one- and two-photon light in random media have been derived in \cite{Kraisler-Schotland-JMP-22, Kraisler-Schotland-JMP-23}, further illustrating the relevance of nonlocal models in complex optical environments. Analytic solution properties and complex-valued formulations of the fractional Helmholtz equation were investigated in \cite{Shen-Zhang-MMAS-25}.

Complementary to these developments on direct problems, inverse problems for partial differential equations involving the fractional Laplacian have attracted significant attention in recent years. A substantial line of work focuses on fractional analogues of classical Calder\'{o}n or Schr\"{o}dinger inverse problems. Results include uniqueness and reconstruction from a single measurement for the fractional Calder\'{o}n problem \cite{Ghosh-Ruland-Salo-Uhlmann-JFA-20} and further developments on stability and low regularity formulations \cite{Ruland-Salo-NA-20, Covi-Railo-Tyni-Zimmermann-SIMA-24, Ruland-SIMA-21, Railo-Zimmermann-NA-24}. For the fractional Schr\"{o}dinger equation, both uniqueness and recovery of potentials have been established in \cite{Ghosh-Salo-Uhlmann-AP-20}, while extensions to magnetic settings are investigated in \cite{Covi-IP-20}. Inverse problems for fractional Helmholtz equations have also been studied. The reconstruction of source and potential was discussed in \cite{Cao-Liu-CMS-19}, whereas geometrical optics constructions and related inverse results were recently developed in \cite{Covi-Hoop-Salo-MMMAS-25}. Bayesian and statistical frameworks for inverse scattering with fractional Helmholtz operators have been explored in \cite{Jia-Yue-Peng-Gao-JFA-18}. Broader surveys and overviews of fractional inverse conductivity problems can be found in \cite{Covi-25}. These works demonstrate the rapidly growing theoretical foundation for inverse problems involving nonlocal operators and highlight the importance of developing new techniques tailored to fractional models.

Inverse problems with random sources or random potentials have been actively studied in recent years. An inverse scattering framework for two-dimensional random potentials was established in \cite{Lassas-Paivarinta-Saksman-CMP-08}, where it was shown that suitable statistical information of the potential can be recovered from ensemble averaged measurements. This direction was further advanced in \cite{Caro-Helin-Lassas-AA-19}, in which inverse scattering for higher dimensional random potentials was analyzed and uniqueness results for recovering the covariance structure were obtained. More recently, simultaneous identification of random potentials and random sources for a Schr\"{o}dinger model was examined in \cite{Li-Liu-Ma-CMP-21}. Inverse random source problems for wave equations have also been investigated: an attenuated Helmholtz model was examined in \cite{Li-Wang-SIAP-21}, and stochastic wave equations with far-field measurements were investigated in \cite{Li-Li-Wang-SIAP-22}. These works demonstrate that meaningful statistical characteristics, such as covariance structures or correlation functions, of random sources or potentials can often be recovered from scattering data, even in the presence of significant randomness and low regularity.

Building on these developments, the present work considers a broad class of highly irregular random sources that align naturally with the nonlocal structure of the fractional Helmholtz operator. Under minimal assumptions on the source regularity, existence, uniqueness, and regularity of the resulting wave field are established through an analysis of the associated Lippmann--Schwinger equation. For the inverse problem, it is shown that key statistical characteristics of the random source, specifically, the principal strengths of its covariance and relation operators, can be uniquely recovered almost surely from far-field data generated by a single realization. Our approach employs the Born approximation, high-frequency expansions of the Green's function associated with the fractional Helmholtz operator, and microlocal methods involving Fourier integral operators. These results extend the existing theory of inverse random source problems to the fractional Helmholtz setting and demonstrate that statistical information remains recoverable even in nonlocal and stochastic wave propagation models.

The paper is organized as follows. Section~\ref{sec2} introduces several preliminaries on weighted Sobolev spaces and GMIG random fields. In Section~\ref{sec3}, the well-posedness of the direct problem is established. Section~\ref{sec4} addresses the inverse problem and establishes a uniqueness theorem for determining the principal symbols of both the covariance and relation operators of the random source. General concluding remarks are provided in Section~\ref{sec5}. The appendix contains the derivation of the Green's function for the fractional Helmholtz equation and its asymptotic behavior.

\section{Preliminaries}\label{sec2}

In this section, we introduce weighted Sobolev spaces and complex-valued GMIG random fields employed in this work.

\subsection{Sobolev spaces}

For $x \in \mathbb{ R }^{ d }$, let $\langle x \rangle := ( 1 + | x |^{ 2 } )^{ 1 / 2 }$. For any $\delta \in \mathbb{R}$ and $1 \le p < \infty$, the weighted $L^{ p }$-norm and the corresponding weighted $L^{ p }$ space on $\mathbb R^{ d }$ are defined by
\begin{align*}
    \| \phi \|_{ L_\delta^p ( \mathbb{R}^{ d } ) } : = \| \langle \cdot \rangle^{ \delta } \phi ( \cdot ) \|_{ L^p ( \mathbb{R}^{ d } ) },\quad L_\delta^p ( \mathbb{R}^{ d } ) : = \big \{ \phi \in L_{ \text{loc} }^1 ( \mathbb{R}^{ d } ) : \| \phi \|_{ L_\delta^p ( \mathbb{R}^{ d } ) }< \infty \big\}.
\end{align*}
For any subset $U\subset\mathbb R^d$, the space $L_\delta^p ( U )$ is defined analogously by replacing $\mathbb{R}^{ d }$ with $U$. 

Let $\mathscr{D} ( \mathbb{R}^{ d } )$ denote the space of test functions on $\mathbb{R}^{d}$, i.e., the space $C^{ \infty }_{ 0 } ( \mathbb{R}^{ d } )$ endowed with the usual locally convex topology. Its topological dual is the space of distributions, denoted by $\mathscr{D}^{ \prime } ( \mathbb{R}^{ d } )$. Let $I$ be the identity operator. We define the weighted Sobolev norms and spaces by
\begin{align*}
    \| \phi \|_{ H_{ \delta }^{ s, p } ( \mathbb{R}^{ d } ) } :  = \|(I-\Delta)^{ s / 2 } \phi \|_{L_\delta^p(\mathbb{R}^{ d })}, 
    \quad H_\delta^{s, p}(\mathbb{R}^{ d }) : = \big\{\phi \in \mathscr{D}^{\prime}(\mathbb{R}^{ d }) : \|\phi\|_{H_\delta^{s, p}(\mathbb{R}^{ d })} < \infty \big\}. 
\end{align*}
We adopt the simplifications $H_{ \delta }^{ s, 2 } ( \mathbb{R}^{ d } ) = H_{ \delta }^{ s } ( \mathbb{R}^{ d } ) $ and $H_{ 0 }^{ s, p } ( \mathbb{R}^{ d } ) = H^{ s, p } ( \mathbb{R}^{ d } )$. It can be verified that 
\begin{align}\label{14}
    \| \phi \|_{ H_\delta^s ( \mathbb{R}^{ d } ) } = ( 2 \pi )^{ - d / 2 } \| \langle \cdot \rangle^{ s } \hat{\phi}(\cdot) \|_{ H^\delta ( \mathbb{R}^{ d } ) }.
\end{align}

\subsection{Complex-valued GMIG random fields}

Let $( \Omega, \mathcal{F}, P )$ be a complete probability space. A complex-valued generalized Gaussian random field $h$ is defined as a mapping $h : \Omega \to \mathscr{D}^{ \prime } ( \mathbb{R}^{ d } )$ such that for any test function $\psi \in \mathscr{D} ( \mathbb{R}^{ d } )$, the dual product $\langle h ( \omega ), \psi \rangle$ defines a complex-valued Gaussian random variable.

For any $\varphi, \psi \in \mathscr{D} ( \mathbb{R}^{ d } )$, the covariance operator $\mathfrak{C}_{ h } : \mathscr{D} ( \mathbb{R}^{ d } ) \to \mathscr{D}^{ \prime } ( \mathbb{R}^{ d } )$ and relation operator $\mathfrak{R}_{ h } : \mathscr{D} ( \mathbb{R}^{ d } ) \to \mathscr{D}^{ \prime } ( \mathbb{R}^{ d } )$ of a centered complex-valued generalized Gaussian random field $h$ are defined respectively by 
\begin{align*}
    ( \mathfrak{ C }_{ h } \varphi ) ( \psi ) : = \mathbb{ E } ( \langle \overline{ h ( \cdot, \omega ) }, \varphi \rangle \langle h ( \cdot, \omega ), \psi \rangle ), \quad ( \mathfrak{R}_{ h } \varphi ) ( \psi ) : = \mathbb{E}(\langle h(\cdot, \omega), \varphi\rangle\langle h(\cdot, \omega), \psi\rangle).
\end{align*}
By the Schwartz kernel theorem, the operators $\mathfrak{C}_{ h }$ and $\mathfrak{R}_{ h }$ admit distributional kernels $K^{ c }_{ h } ( x, y )$ and $K^{ r }_{ h } ( x, y )$, respectively, such that, for all $\varphi, \psi \in \mathscr{D} ( \mathbb{R}^{ d } )$, 
\begin{align}\label{47}
    ( \mathfrak{C}_{ h } \varphi ) ( \psi ) = \langle K_{ h }^{ c }, \varphi \otimes \psi \rangle, \quad ( \mathfrak{R}_{ h } \varphi ) ( \psi ) = \langle K_{ h }^{ r }, \varphi \otimes \psi \rangle.
\end{align}

A complex-valued centered generalized Gaussian random field $h$ on $\mathbb{R}^{ d }$ is termed microlocally isotropic of order $- m$ in an open set $D \subset \mathbb{ R }^ { d }$ when both its covariance operator $\mathfrak{C}_{ h }$ and relation operator $\mathfrak{R}_{ h }$ are classical pseudo-differential operators of order $ - m $. The corresponding symbols $\sigma^{ c }_{ h }, \sigma^{ r }_{ h }\in S^{-m}(\mathbb R^d\times\mathbb R^d)$ satisfy 
\begin{align*}
    \sigma^{ c }_{ h } = \mu^{ c }(x) | \xi |^{ - m } + O(|\xi|^{-m-1}), \quad \sigma^{ r }_{ h } = \mu^{ r } ( x ) | \xi |^{ - m } + O(|\xi|^{-m-1}),
\end{align*}
with smooth, compactly supported, nonnegative coefficients $\mu^{ c }, \mu^{ r } \in C^{ \infty }_{ 0 } ( D ) $. The symbol class $S^{ - m } ( \mathbb{R}^{ d } \times \mathbb{R}^{ d } )$ consists of all smooth functions satisfying
\begin{align*}
 |( D_{ x }^{ \alpha } D_{ \xi }^{ \beta } \sigma ) ( x, \xi ) | \leq C_{ \alpha, \beta } ( 1 + | \xi | )^{ - m - | \beta | }
\end{align*} 
for every pair of multi-indices $\alpha, \beta$.

For $\eta \in \{ c, r\}$, the kernels and symbols satisfy the following relations in the sense of distributions:
\begin{align}\label{46}
    K_{ h }^{ \eta } ( x, y ) & = ( 2 \pi )^{ - d } \int_{ \mathbb{R}^{ d } } e^{ \mathrm{i} ( x - y ) \cdot \xi } \sigma^{ \eta }_{ h } ( x, \xi ) d \xi,
    \notag \\ \sigma^{ \eta }_{ h } ( x, \xi ) & = \int_{ D } e^{ - \mathrm{i} \xi \cdot ( x - y ) } K^{ \eta }_{ h } ( x, y ) d y.
\end{align}

The following lemma describes the regularity of GMIG random fields (cf. \cite[Lemma 2.6]{Li-Wang-SIAP-21}).
 
\begin{lemma}\label{63}
Let $h$ be a GMIG field of order $ - m $ supported in an open set $D \subset \mathbb{R}^{ d }$, with $m<d$. Then, for any $\epsilon > 0$ and any $p \in ( 1, \infty )$, it holds almost surely that $h \in W^{ - \frac{ d - m }{ 2 } - \epsilon, p } ( D ) $.
\end{lemma}

For the inverse problem, we impose the following assumptions on the random source $f$ and the deterministic potential $q$.

\begin{assumption}\label{117}
Let the source $f$ be a complex-valued centered GMIG random field of order $ -m $ satisfying $m \in ( m_{ d, \alpha }, d )$ and $\alpha \in ( \alpha_{ d }, 1 )$, where
\begin{align*}
    m_{ d, \alpha } = 
    \begin{cases}
        \frac{ 3 }{ 2 } - 2 \alpha, & d = 1,
        \\[3pt] \frac{ 5 }{ 2 } - 2 \alpha, & d = 2,
        \\[3pt] \frac{ 15 }{ 4 } - 2 \alpha, & d = 3, 
    \end{cases}
    \quad \alpha_{ d } = 
    \begin{cases}
        \frac{ 1 }{ 4 }, & d = 1, 
        \\[3pt] \frac{ 1 }{ 4 }, & d = 2, 
        \\[3pt] \frac{ 3 }{ 8 }, & d = 3.
    \end{cases}
\end{align*}
The field $f$ is compactly supported in a convex bounded region $D \subset \mathbb{R}^{ d }$. Its covariance and relation operators have principal symbols $\mu^{ c } ( x ) | \xi |^{ - m }$ and $\mu^{ r } ( x ) | \xi |^{ - m }$, where $\mu^{ c }, \mu^{ r } \in C_0^{\infty} ( D )$. The potential $q \in H_{ 0 }^{ N_{ \alpha } } ( \mathbb{R}^{ d })$ is a deterministic function satisfying $\operatorname{supp} q \subset U \subset \mathbb{R}^{ d }$, where $U$ is a convex bounded domain with $\operatorname{dist} ( D, U ) > 0$, and
\begin{align*}
    N_{ \alpha } = N_{ \alpha } ( d, m ) = 
    \begin{cases}
        0, & d = 1, 
        \\ \lceil \frac{ 5 }{ 2 } + m - 2 \alpha \rceil, & d = 2, 
        \\ \lceil 3 + \frac{ m }{ 2 } \rceil, & d = 3.
    \end{cases}
\end{align*}
The notation $\lceil x \rceil$ denotes the ceiling of $x$, which is the smallest integer greater than or equal to $x$.
\end{assumption}

\begin{remark}
Since both $D$ and $U$ are convex, there exists a hyperplane in $\mathbb{R}^{ d }$ that separates them. Let $\hat{ n }$ denote the unit normal vector of such a hyperplane, chosen to point from the half-space containing $U$ toward the half-space containing $D$.
\end{remark}

\section{The direct problem}\label{sec3}

In this section, we establish the well-posedness of the direct problem, demonstrating that it admits a unique solution almost surely in a suitable weighted Sobolev space, interpreted in the distributional sense.

We begin by introducing the fundamental solution $G^{ k }$ to the fractional Helmholtz equation, defined as the solution of
\begin{align*}
    \begin{cases}
        ( - \Delta )^{ \alpha } G^{ k } ( x ) - k^{ 2 \alpha } G^{ k } ( x  ) = \delta ( x ) & \text{in} \ \mathbb{R}^{ d }, 
        \\ \lim \limits_{r \rightarrow \infty} r^{ \frac{ d - 1 }{ 2 } } ( \partial_{ r } G^{ k } - \mathrm{i} k G^{ k } ) = 0 & \text{for} \ r = | x |.
    \end{cases}
\end{align*}
The explicit expression for the Green's function $G^k$ is given by
\begin{align*}
    G^{ k } ( x ) = 
    \begin{cases}
    - \frac{ \sqrt{ \pi } k^{ 1 - 2 \alpha } }{ 4 \alpha } H^{ 2, 1 }_{ 2, 4 } ( \frac{ k | x | }{ 2 } ) + \frac{ \mathrm{i} k^{ 1 - 2 \alpha } }{ 2 \alpha } \cos ( k | x | ),  & d = 1,\\[4pt]
    - \frac{ k^{ 2 - 2 \alpha } }{ 8 \alpha } H^{ 2, 1 }_{ 2, 4 } ( \frac{ k | x | }{ 2 } ) + \frac{ \mathrm{i} k^{ 2 - 2 \alpha } }{ 4 \alpha } J_{ 0 } ( k | x | ), & d = 2,\\[4pt] 
    - \frac{ k^{ 3 - 2 \alpha } }{ 16 \alpha \sqrt{ \pi } } H^{ 2, 1 }_{ 2, 4 } ( \frac{ k | x | }{ 2 } ) + \frac{ \mathrm{i} k^{ 2 - 2 \alpha } }{ 4 \pi \alpha } \frac{ \sin ( k | x | ) }{ | x | }, & d = 3.
    \end{cases}
\end{align*}
A complete derivation of the expression for $G^k$ is provided in the Appendix.

Define the resolvent operator $\mathcal{ H }_{ k }$ by 
\begin{align*}
    ( \mathcal{H}_k \varphi ) ( x ) : = \int_{ \mathbb{R}^{ d } } G^k ( x, z ) \varphi ( z ) d z. 
\end{align*}
This operator satisfies the following quantitative estimate, which makes explicit the dependence on the wavenumber $k$. The proof is inspired by \cite{Li-Liu-Ma-CMP-21}.

\begin{theorem}\label{6}
Let $s \in [ \frac{ \alpha }{ 2 }, \alpha )$, $\epsilon > 0$, and $p \in ( 1, \infty )$. If $k > 2^{ \frac{1}{ 2 \alpha - 2 s } }$, then there exists a constant $C_{ \alpha, s, p, d, \epsilon }>0 $ depending on $\alpha, s, p, d$, and $\epsilon$ such that
\begin{align*}
    \| \mathcal{ H }_{ k } \varphi \|_{ H_{ - 1 / 2 - \epsilon }^{ s / p } ( \mathbb{R}^{ d } ) } \leq C_{ \alpha, s, p, d, \epsilon } k^{ - 2 s ( 1 - \frac{ 1 }{ p } ) } \| \varphi \|_{ H_{ 1 / 2 + \epsilon }^{ - s / p }( \mathbb{ R }^{ d } ) }, \quad \forall \, \varphi \in H_{ 1 / 2 + \epsilon }^{ - s / p } ( \mathbb{ R }^{ d } ).
\end{align*}
\end{theorem}

\begin{proof}
For any $\tau>0$, define an operator
\begin{align*}
    \mathcal{H}_{k, \tau} \varphi(x) := (2 \pi)^{ - d } \int_{ \mathbb{R}^{ d } } e^{\mathrm{i} x \cdot \xi} \frac{ \hat{\varphi} ( \xi ) }{ | \xi |^{ 2 \alpha } - k^{ 2 \alpha } - \mathrm{i} \tau } d \xi. 
\end{align*}
Let $\chi \in C_{ 0 }^{\infty}(\mathbb{R}^{ d })$ be a cutoff function satisfying $\chi ( x ) = 1$ for $| x | \leq 1$ and $\chi ( x ) = 0$ for $| x | \geq 2$. For any $p \in(1,+\infty)$, we obtain
\begin{align*}
    & ( 2 \pi )^{ d } ( \mathcal{H}_{ k, \tau } \varphi, \psi )_{ L^{ 2 } ( \mathbb{R}^{ d } ) } = ( 2 \pi )^{ d } \int_{ \mathbb{R}^{ d } } \mathcal{H}_{ k, \tau } \varphi ( x ) \overline{ \psi ( x ) } d x = \int_{ \mathbb{R}^{ d } } \mathcal{F} \big[ \mathcal{H}_{k, \tau} \varphi \big] (\xi)  \overline{\mathcal{F} \big[ \psi \big] } (\xi) d \xi
    \\ & = \int_{ 0 }^{ \infty } \frac{ 1 - \chi^{ 2 } ( r^{ 2 \alpha  - 2s } - k^{ 2 \alpha - 2s } ) } { r^{ 2 \alpha } - k^{ 2 \alpha } - \mathrm{i} \tau } d r \int_{ | \xi | = r } \hat{ \varphi } ( \xi ) \overline{ \hat{ \psi } } ( \xi ) d S ( \xi ) 
    \\ & \quad + \int_0^{\infty} \frac{ \chi^2( r^{ 2 \alpha - 2s } - k^{ 2 \alpha - 2s } ) }{ r^{ 2 \alpha } - k^{ 2 \alpha } - \mathrm{i} \tau} d r  \int_{ | \xi | = r } \hat{ \varphi } ( \xi ) \overline{ \hat{ \psi } } ( \xi ) d S ( \xi )
    \\ & = \int_{ 0 }^{ \infty } \frac{ 1 - \chi^2( r^{ 2 \alpha - 2s } - k^{ 2 \alpha - 2s } ) }{ r^{ 2 \alpha } - k^{ 2 \alpha } - \mathrm{i} \tau } d r \int_{ | \xi | = r } \hat{ \varphi } ( \xi ) \overline{ \hat{ \psi } } ( \xi ) d S ( \xi )
    \\ & \quad + \int_0^{\infty} \frac{\langle r\rangle^{ \frac{ 2s }{ p } } r^{ d - 1 } \chi^2( r^{ 2 \alpha - 2s } - k^{ 2 \alpha - 2s } ) }{ r^{ 2 \alpha } - k^{ 2 \alpha } -\mathrm{i} \tau } d r  \int_{ \mathbb{S}^{ d - 1 } } \langle k\rangle^{ - \frac{ s }{ p } } \hat{\varphi}(k \omega)\langle k\rangle^{ - \frac{ s }{ p } } \overline{ \hat{ \psi } } ( k \omega) d S(\omega) 
    \\ & \quad + \int_0^{\infty} \frac{\langle r\rangle^{ \frac{ 2s }{ p } } r^{ d - 1 } \chi^2( r^{ 2 \alpha - 2s } - k^{ 2 \alpha - 2s } ) }{ r^{ 2 \alpha } - k^{ 2 \alpha } - \mathrm{i} \tau} d r \int_{ \mathbb{S}^{ d - 1 }} \Big[ \langle r\rangle^{ - \frac{s}{p} } \hat{\varphi}(r \omega) \langle r\rangle^{ - \frac{s}{p} } \overline{ \hat{ \psi } } ( r \omega ) 
    \\ & \quad - \langle k\rangle^{ - \frac{s}{p} } \hat{\varphi}(k \omega) \langle k\rangle^{ - \frac{s}{p} } \overline{ \hat{ \psi } } ( k \omega ) \Big] d S(\omega) 
    \\ & = : I_{ 1 } ( \tau ) + I_{ 2 } ( \tau ) + I_{ 3 } ( \tau ).
\end{align*}

We next estimate the terms $I_{ j } ( \tau )$, $j = 1, 2, 3$, beginning with $I_{ 1 }( \tau )$. Recall Young's inequality,
\begin{align}\label{15}
    a b \leq \frac{a^{p}}{p} + \frac{b^{q}}{q} \quad \Rightarrow \quad p^{\frac{1}{p}} q^{\frac{1}{q}} a^{\frac{1}{p}} b^{\frac{1}{q}} \leq a + b, 
\end{align}
for any $a, b > 0$ and $p, q > 1$ satisfying $\frac{ 1 }{ p } + \frac{ 1 }{ q } = 1$. For $s \in [ \frac{ \alpha }{ 2 }, \alpha )$, whenever $| r^{ 2 \alpha - 2s } - k^{ 2 \alpha - 2s } | > 1$, we have
\begin{align}\label{137}
    \frac{ 1 }{ | r^{ 2 \alpha } - k^{ 2 \alpha } - \mathrm{i} \tau | } \leq \frac{ 1 }{ | r^{ 2 \alpha } - k^{ 2 \alpha } | } \leq \frac{ 1 }{ ( r^{ 2 s } + k^{ 2 s } ) | r^{ 2 \alpha - 2 s } - k^{ 2 \alpha - 2 s } | } \leq \frac{ 1 }{ r^{ 2 s } + k^{ 2 s } }.
\end{align}
We observe that $I_{ 1 } \neq 0$ only in the region where $| r^{ 2 \alpha - 2s } - k^{ 2 \alpha - 2s } | > 1$. Using Young’s inequality \eqref{15} together with \eqref{137}, and assuming $k > 1$, we obtain for any $1 < p < \infty$ and $\delta \geq 0$ that 
\begin{align}\label{18}
    | I_{ 1 } ( \tau ) | & = \Big| \int_{ 0 }^{ \infty } \frac{ 1 - \chi^{ 2 }( r^{ 2 \alpha - 2s } - k^{ 2 \alpha - 2s } ) }{ r^{ 2 \alpha } - k^{ 2 \alpha } - \mathrm{i} \tau} d r  \int_{ | \xi | = r } \hat{ \varphi } ( \xi ) \overline{ \hat{ \psi } } (\xi) d S ( \xi ) \Big|
    \notag \\ & \leq \int_{ 0 }^{ \infty } \frac{ 1 - \chi^{ 2 } ( r^{ 2 \alpha - 2s } - k^{ 2 \alpha - 2s } ) }{ r^{ 2 s } + k^{ 2 s } } d r  \int_{ | \xi | = r } | \hat{ \varphi } ( \xi ) | | \hat{ \psi } ( \xi ) | d S ( \xi ) 
    \notag \\ & \leq \int_0^{\infty} \frac{1 - \chi^2( r^{ 2 \alpha - 2 s } - k^{ 2 \alpha - 2 s } ) }{ ( r^{ 2 s } + 1) + ( k^{ 2 s } - 1 ) } d r  \int_{ | \xi | = r } | \hat{ \varphi } ( \xi ) | | \hat{ \psi } ( \xi ) | d S ( \xi ) 
    \notag \\ & \leq C_{p} \int_0^{\infty} \frac{1 - \chi^2( r^{ 2 \alpha - 2 s } - k^{ 2 \alpha - 2 s } ) }{ ( r^{ 2 s } + 1)^{ \frac{ 1 }{ p } }  ( k^{ 2 s } - 1)^{\frac{1}{q}} } d r \int_{ | \xi | = r } | \hat{ \varphi } ( \xi ) | | \hat{ \psi } ( \xi ) | d S ( \xi )  
    \notag \\ & \leq C_p k^{ -\frac{ 2s }{ q } } \int_0^{\infty} \langle r \rangle^{ - \frac{ 2s }{ p } } d r \int_{ | \xi | = r } | \hat{ \varphi } ( \xi ) | | \hat{ \psi } ( \xi ) | d S ( \xi ) 
    \notag \\ & \leq C_p k^{ - 2s ( 1 - \frac{1}{p} ) }\|\varphi\|_{ H_\delta^{ - s / p } ( \mathbb{R}^{ d } ) } \|\psi\|_{ H_{ \delta }^{ - s / p }(\mathbb{R}^{ d })},
\end{align}
where the constant $C_p$ is independent of $\tau$.

Next, we estimate $I_{ 2 } ( \tau )$. It can be directly observed that
\begin{align}\label{9}
    I_{ 2 } ( \tau ) = \int_{ \mathbb{S}^{ d - 1 } } \langle k\rangle^{ - \frac{ s }{ p } } \hat{\varphi}(k \omega) \langle k\rangle^{ - \frac{ s }{ p } } \overline{ \hat{ \psi } } ( k \omega ) \int_{ 0 }^{ \infty } \frac{\langle r\rangle^{ \frac{ 2s }{ p } } r^{ d - 1 } \chi^2( r^{ 2 \alpha - 2s } - k^{ 2 \alpha - 2s } ) }{ r^{ 2 \alpha } - k^{ 2 \alpha } -\mathrm{i} \tau } d r d S ( \omega ).
\end{align}
For some $\tau_{ 0 } \in ( 0, 1 )$ and $\tau \in ( 0, \tau_{ 0 } )$, we define $p_{ \tau } ( r ) : = p ( r ) = r^{ 2 \alpha } - k^{ 2 \alpha } - \mathrm{i} \tau$. Recall that $\chi^2( r^{ 2 \alpha - 2 s } - k^{ 2 \alpha - 2 s } ) = 0$ whenever $| r^{ 2 \alpha - 2 s } - k^{ 2 \alpha - 2 s } | > 2$. For any $r \in \{ r \in \mathbb{R} : \tau_0 \leq | r^{ 2 \alpha - 2 s } - k^{ 2 \alpha - 2 s } | \leq 2 \}$ and any $s \in [ \frac{ \alpha }{ 2 }, \alpha )$, we get
\begin{align}\label{7}
    | p ( r ) | \geq | r^{ 2 \alpha - 2 s } - k^{ 2 \alpha - 2 s } | ( r^{ 2 s } + k^{ 2 s } ) \geq \tau_{ 0 } ( r^{ 2 s } + k^{ 2 s } ) \geq \tau_{ 0 } k^{ 2s }.
\end{align}
When $k > 2^{ \frac{ 1 }{ 2 \alpha - 2 s } }$, there exist two positive numbers $r_1<k<r_2$ such that $| r_{ j }^{ 2 \alpha - 2 s } -  k^{ 2 \alpha - 2 s } | = \tau_{ 0 }$, $j = 1, 2$. Define
\begin{align*}
    \Gamma_{ k, \tau_{ 0 } } & : = \big\{r \in \mathbb{C} : r = r_{ 1 } - \mathrm{i} \gamma, \, \gamma \in [0, \gamma_{0} ] \big\} \cup \big\{r \in \mathbb{C} : r = \tilde{r} - \mathrm{i} \gamma_{0}, \, \tilde{r} \in [r_{1}, r_{2}] \big\}
    \\ & \quad \cup \big\{r \in \mathbb{C} : r = r_{ 2 } - \mathrm{i} \gamma, \, \gamma \in [0, \gamma_{0} ] \big\},
\end{align*}
where $\gamma_{ 0 } : = \max \{ ( r_{ 2 }^{ 2 } - k^{ 2 } )^{ \frac{ 1 }{ 2 } }, ( k^{ 2 } - r_{ 1 }^{ 2 } )^{ \frac{ 1 }{ 2 } } \}$.

We choose $\tau_{ 0 }$ sufficiently small such that for all $r \in \Gamma_{ k, \tau_{ 0 } }$ and $\alpha \in ( 0, 1 )$, $\Im ( r ) < 0$. So, by (\ref{7}), we have
\begin{align*}
    | ( k - \mathrm{ i } \gamma_{ 0 } )^{ 2 \alpha } - k^{ 2 \alpha } | & \geq | k - \mathrm{ i } \gamma_{ 0 } |^{ 2 \alpha } - k^{ 2 \alpha } \geq ( k^{ 2 } + \gamma_{ 0 }^{ 2 } )^{ \alpha } - k^{ 2 \alpha } 
    \\& \geq r_{ 2 }^{ 2 \alpha } - k^{ 2 \alpha } \geq \tau_{ 0 } k^{ 2s }.
\end{align*}
Then for any $r \in \Gamma_{ k, \tau_{ 0 } }$ and $\tau \in ( 0, \tau_{ 0 } )$, by \eqref{7}, we obtain that 
\begin{align}\label{8}
    | p ( r ) | & = | r^{ 2 \alpha } - k^{ 2 \alpha } - \mathrm{i} \tau | = | \Re ( r^{ 2 \alpha } ) - k^{ 2 \alpha } - \mathrm{i} ( \tau - \Im ( r^{ 2 \alpha } ) ) | \geq | r^{ 2 \alpha } - k^{ 2 \alpha } | 
    \notag \\ & \geq \min \{ r_{ 2 }^{ 2 \alpha } - k^{ 2 \alpha }, k^{ 2 \alpha } - r_{ 1 }^{ 2 \alpha }, | ( k - \mathrm{ i } \gamma_{ 0 } )^{ 2 \alpha } - k^{ 2 \alpha } | \} \geq \tau_{ 0 } k^{ 2 s }.
\end{align}
Combining (\ref{7}) and (\ref{8}), for any $ \tau \in (0, \tau_0)$, any $k > 2^{ \frac{ 1 }{ 2 \alpha - 2 s } }$, and any $r \in \{ r \in \mathbb{R}_{+} : \tau_{ 0 } \leq | r^{ 2 \alpha - 2 s } - k^{ 2 \alpha - 2 s } | \leq 2 \} \cup \Gamma_{ k, \tau_{ 0 } }$, we get 
\begin{align}\label{10}
    | p_{ \tau } ( r ) | \geq \tau_{ 0 } k^{ 2 s }.
\end{align}
Moreover, for $k > 2^{ \frac{ 1 }{ 2 \alpha - 2 s } }$ and $r \in \{r \in \mathbb{R}_{ + } :  \tau_{ 0 } \leq | r^{ 2 \alpha - 2 s } - k^{ 2 \alpha - 2 s } | \leq 2 \}$, we have
\begin{align}\label{138}
    \langle r \rangle^{ \frac{ 2 s }{ p } } \lesssim \langle k \rangle^{ \frac{ 2 s }{ p } }, \quad 0 < \Big( \frac{ r }{ k } \Big)^{ d - 1 } < C ( d ),
\end{align}
where the notation $a\lesssim b$ means that $a\leq C b$  for some constant $C>0$ independent of the wavenumber $k$, and the constant $C ( d ) > 0$ depends only on the spatial dimension. 

By means of Cauchy's integral theorem, we deform the integration contour for $r$ in (\ref{9}) from $\mathbb{R}_{ + }$ to $\{ r \in \mathbb{R}_{+} : \tau_{ 0 } \leq | r^{ 2 \alpha - 2 s } - k^{ 2 \alpha - 2 s } | \leq 2 \} \cup \Gamma_{k, \tau_0}$.  Applying \eqref{10} and \eqref{138}, we obtain for all  $\tau \in ( 0, \tau_{ 0 } )$ and $k > 2^{ \frac{ 1 }{ 2 \alpha - 2 s } }$,
\begin{align}\label{12}
     | I_{ 2 } ( \tau ) | & \leq \int_{|\xi|=k} \langle \xi \rangle^{ - \frac{s}{p} } |\hat{\varphi}(\xi)|  \langle\xi\rangle^{ - \frac{s}{p} } |\hat{\psi}(\xi)| \int_{ \{ r \in \mathbb{R}_{ + } : \tau_{ 0 } \leq | r^{ 2 \alpha - 2 s } - k^{ 2 \alpha - 2 s } | \leq 2 \} } \frac{ \langle r \rangle^{ 2s / p } ( r / k )^{ d - 1 } }{ \tau_{ 0 } k^{ 2 s } } d r  d S( \xi ) 
     \notag \\ & \quad + \int_{ | \xi | = k } \langle \xi \rangle^{ - \frac{s}{p} } | \hat{ \varphi } ( \xi ) | \langle \xi \rangle^{ - \frac{ s }{ p } } |\hat{ \psi } ( \xi ) | \int_{ \Gamma_{ k, \tau_{ 0 } } } \frac{ ( 1 + |r|^{2} )^{ s / p } ( | r | / k )^{ d - 1 } }{ \tau_0 k^{ 2s } } d r d S(\xi). 
\end{align}

Note that for any $r \in \Gamma_{ k, \tau_{ 0 } }$, 
\begin{align*}
    | r |^{2} & \leq r_{ 2 }^{ 2 } + \gamma_{ 0 }^{ 2 } \lesssim k^{ 2 }.
\end{align*}
Hence, for all $r \in \Gamma_{ k, \tau_{ 0 } }$, 
\begin{align}\label{11}
    ( 1 + | r |^{ 2 } )^{ \frac{ s }{ p } } \lesssim \langle k \rangle^{ \frac{ 2 s }{ p } }, \quad \big( \frac{ | r | }{ k } \big)^{ d - 1 } \lesssim 1.
\end{align}
Using (\ref{14}), (\ref{138}), (\ref{11}), and \cite[Remark 13.1]{Eskin-11}, we can proceed from (\ref{12}) to the following bound:
\begin{align}\label{13}
    | I_{ 2 } ( \tau ) | & \leq C_{ \alpha, s, \tau_0, p, d } \int_{ | \xi | = k } \langle \xi \rangle^{ - \frac{ s }{ p } } | \hat{ \varphi } ( \xi ) | \langle \xi \rangle^{ - \frac{ s }{ p } } | \hat{ \psi } ( \xi ) |
    \notag \\ & \quad \times \int_{ \Gamma_{ k, \tau_{ 0 } } \cup \{ r \in \mathbb{ R }_{ + } : \tau_{ 0 } \leq | r^{ 2 \alpha - 2 s } - k^{ 2 \alpha - 2 s } | \leq 2 \} } \frac{\langle k\rangle^{ \frac{ 2s }{p} }}{\tau_0 k^{ 2s } } d r d S(\xi) 
    \notag \\ & \leq C_{ \alpha, s, \tau_0, p, d, \epsilon } k^{ - 2s ( 1 - \frac{1}{p}  ) } \|\langle\cdot\rangle^{ - \frac{s}{p} } \hat{\varphi}(\cdot)\|_{ H^{ \frac{1}{2} + \epsilon } ( \mathbb{R}^{ d } ) } \|\langle\cdot\rangle^{ - \frac{s}{p} } \hat{\psi}(\cdot)\|_{ H^{ \frac{1}{2} + \epsilon } ( \mathbb{R}^{ d } ) } 
    \notag \\ & \leq C_{ \alpha, s, \tau_0, p, d, \epsilon } k^{ - 2s ( 1 - \frac{1}{p}  ) } \|\varphi\|_{ H_{ 1 / 2 + \epsilon }^{ - s / p } ( \mathbb{R}^{ d } ) } \|\psi\|_{H_{ 1 / 2 + \epsilon }^{ - s / p } ( \mathbb{R}^{ d } ) },
\end{align}
where the constant $C_{ \alpha, s, \tau_0, p, d, \epsilon } > 0$ is independent of $\tau$.  It is worth mentioning that the appearance of the arbitrarily small positive number $\epsilon$ in the weighted Sobolev norm $\| \cdot \|_{ H^{ 1 / 2 + \epsilon } }$ in (\ref{13}) is required because the Sobolev order must be strictly larger than $1 / 2$ (cf. \cite[Remark 13.1]{Eskin-11}).

Finally, we turn to the estimate of $I_{ 3 } ( \tau ) $. Let $\mathbb{F}_{ r } ( \omega ) : = \hat{ \varphi } ( r \omega )$, $\mathbb{G}_{ r } ( \omega ) : = \overline{ \hat{ \psi } } ( r \omega )$, and $\mathbb{S}_{ r }^{ d - 1 }$ be the centered sphere of radius $r$ in $\mathbb{R}^{ d }$. A direct computation yields
\begin{align*}
    | I_{ 3 } ( \tau ) | & = \Big| \int_0^{ \infty } \frac{\langle r\rangle^{ \frac{ 2s }{ p } } r^{ d - 1 } \chi^2( r^{ 2 \alpha - 2 s } - k^{ 2 \alpha - 2s } ) } { r^{ 2 \alpha } - k^{ 2 \alpha } - \mathrm{i} \tau } d r \int_{\mathbb{S}^{ d - 1 } } \Big( \langle r\rangle^{ - \frac{ 2s }{ p } } \mathbb{F}_r \mathbb{G}_r - \langle k \rangle^{ - \frac{ 2s}{ p } } \mathbb{F}_k \mathbb{G}_k \Big) d S(\omega) \Big| 
    \\ & \leq \Big| \int_0^{\infty} \frac{\langle r\rangle^{ \frac{ 2s }{ p } } r^{ d - 1 } \chi^2( r^{ 2 \alpha - 2 s } - k^{ 2 \alpha - 2 s } ) } { r^{ 2 \alpha } - k^{ 2 \alpha } - \mathrm{i} \tau } d r \int_{ \mathbb{S}^{ d - 1 } } \langle r\rangle^{ - \frac{ s }{ p } } \mathbb{F}_r \Big( \langle r\rangle^{ - \frac{ s }{ p } } \mathbb{G}_r - \langle k \rangle^{ - \frac{ s }{ p } } \mathbb{G}_k \Big) d S(\omega) \Big| 
    \\ & \quad + \Big| \int_0^{\infty} \frac{\langle r\rangle^{ \frac{ 2s }{ p } } r^{ d - 1 } \chi^2( r^{ 2 \alpha - 2 s } - k^{ 2 \alpha - 2 s } ) } { r^{ 2 \alpha } - k^{ 2 \alpha } - \mathrm{i} \tau } d r \int_{\mathbb{S}^{ d - 1 } } \langle k\rangle^{ - \frac{ s }{ p } } \mathbb{G}_{k} \Big( \langle r \rangle^{ - \frac{ s }{ p } } \mathbb{F}_r - \langle k \rangle^{ - \frac{ s }{ p } } \mathbb{F}_k \Big) d S(\omega) \Big| 
    \\ & \leq \int_{ 0 }^{ \infty } \frac{ \langle r \rangle^{ \frac{ 2s }{ p } } \chi^{ 2 } ( r^{ 2 \alpha - 2 s } - k^{ 2 \alpha - 2 s } ) }{ | r^{ 2 \alpha } - k^{ 2 \alpha } | }  \| \langle \cdot \rangle^{ - \frac{ s }{ p } } \hat{ \varphi } ( \cdot ) \|_{ L^{ 2 } ( \mathbb{S}^{ d - 1 }_{ r } ) } r^{ \frac{ d - 1 }{ 2 } } 
    \\ & \quad \times\Big( \int_{ \mathbb{S}^{ d - 1 } }  | \langle r \rangle^{ - \frac{ s }{ p } } \mathbb{G}_r - \langle k \rangle^{ - \frac{ s }{ p }  } \mathbb{G}_k |^2  d S(\omega) \Big) ^{\frac{1}{2}} d r + \int_0^{\infty} \frac{\langle r \rangle^{ \frac{ 2s }{ p } } \chi^2(  r^{ 2 \alpha - 2 s } - k^{ 2 \alpha - 2 s } ) }{ | r^{ 2 \alpha } - k^{ 2 \alpha } | } r^{ \frac{ d - 1 }{ 2 } } 
    \\ & \quad \times \Big( \int_{\mathbb{S}^{ d - 1 } } | \langle r \rangle^{ - \frac{ s }{ p } } \mathbb{F}_r - \langle k \rangle^{ - \frac{ s }{ p } } \mathbb{F}_k |^2 d S(\omega) \Big)^{\frac{1}{2}} r^{ \frac{ d - 1 }{ 2 } } k^{ - \frac{ d - 1 }{ 2 } } \| \langle \cdot \rangle^{ - \frac{ s }{ p } } \hat{ \psi } ( \cdot ) \|_{ L^{ 2 } ( \mathbb{S}_{ k }^{ d - 1 } ) } d r.
\end{align*}
Using \cite[Remark 13.1]{Eskin-11}, \cite[(13.28)]{Eskin-11}, Young's inequality (\ref{15}), and the Fourier characterization (\ref{14}), we obtain
\begin{align}\label{19}
    | I_{ 3 } ( \tau ) | & \leq C_{ s, d, \alpha, \beta, \epsilon, p } \int_{ 0 }^{ \infty } \frac{ \langle r \rangle^{ \frac{ 2s }{ p } } \chi^{ 2 } ( r^{ 2 \alpha - 2 s } - k^{ 2 \alpha - 2 s } ) }{ | r^{ 2 \alpha - 2 s } - k^{ 2 \alpha - 2s  } | ( r^{ 2 s } + k^{ 2 s } ) }  \| \langle \cdot \rangle^{ - \frac{ s }{ p } } \hat{\varphi} ( \cdot ) \|_{ H^{ \frac{ 1 }{ 2 } + \epsilon }(\mathbb{R}^{ d })} \notag\\ 
    &\quad \times |r-k|^{ \beta }  \| \langle \cdot \rangle^{ - \frac{ s }{ p } } \hat{\psi} ( \cdot ) \|_{H^{ \frac{1}{2} + \epsilon } ( \mathbb{R}^{ d } ) } d r\notag \\
    & \leq C_{ s, d, \alpha, \beta, \epsilon, p } \int_0^{\infty} \frac{\langle r\rangle^{ \frac{ 2s }{ p } } \chi^2( r^{ 2 \alpha - 2 s } - k^{ 2 \alpha - 2 s } ) }{ ( r^{ 2 s } + 1 )^{ \frac{1}{p} }( k^{ 2 s } - 1 )^{ 1 -\frac{1}{p} }}  \frac{ | r - k |^{ \beta } }{ | r^{ 2 \alpha - 2 s } - k^{ 2 \alpha - 2 s } | } d r \notag \\ 
    & \quad \times \| \langle \cdot \rangle^{ - \frac{ s }{ p } } \hat{\varphi} ( \cdot ) \|_{ H^{ \frac{ 1 }{ 2 } + \epsilon }(\mathbb{R}^{ d })} \| \langle \cdot \rangle^{ - \frac{ s }{ p } } \hat{\psi} ( \cdot ) \|_{H^{ \frac{1}{2} + \epsilon }(\mathbb{R}^{ d })}\notag \\
    & \leq C_{ s, d, \alpha, \beta, \epsilon, p } k^{ - 2s ( 1 -\frac{1}{p} ) } \int_0^{\infty} \chi^2( r^{ 2 \alpha - 2 s } - k^{ 2 \alpha - 2 s } )\frac{ | r - k |^{ \beta } }{ | r^{ 2 \alpha - 2 s } - k^{ 2 \alpha - 2 s } | } d r \notag\\
    &\quad \times \| \varphi \|_{ H^{ - s / p }_{ 1 / 2 + \epsilon }(\mathbb{R}^{ d })}  \| \psi \|_{H^{ - s / p }_{ 1 / 2 + \epsilon }(\mathbb{R}^{ d })}\notag \\ 
    & \leq C_{ s, d, \alpha, \beta, \epsilon, p } k^{ - 2s ( 1 -\frac{1}{p} ) } \| \varphi \|_{H^{ - s / p }_{ 1 / 2 + \epsilon } ( \mathbb{R}^{ d } ) }  \| \psi \|_{ H^{ - s / p }_{ 1 / 2 + \epsilon } ( \mathbb{R}^{ d } ) },
\end{align}
where $\epsilon > 0$ and $\beta \in ( 0, \epsilon )$. 

Combining estimates (\ref{18}), (\ref{13}) and (\ref{19}) gives
\begin{align*}
    | ( \mathcal{H}_{ k, \tau } \varphi, \psi )_{ L^{ 2 } ( \mathbb{R}^{ d } ) } | \leq C k^{ - 2s ( 1 - \frac{1}{p} )} \|\varphi\|_{H_{ 1 / 2 + \epsilon }^{ - s / p }(\mathbb{R}^{ d } ) } \|\psi\|_{ H_{ 1 / 2 + \epsilon }^{ - s / p }( \mathbb{R}^{ d } ) }, 
\end{align*}
which in turn implies
\begin{align}\label{25}
    \| \mathcal{H}_{ k, \tau } \varphi \|_{ H_{ - 1 / 2 - \epsilon }^{ s / p } ( \mathbb{R}^{ d } ) } \leq C k^{ - 2s ( 1 - \frac{1}{p} ) } \|\varphi\|_{ H_{ 1 / 2 + \epsilon }^{ - s / p } ( \mathbb{R}^{ d } ) },
\end{align}
where the constant $C > 0$ is independent of $\tau$.

Subsequently, we investigate the asymptotic behavior of $\mathcal{H}_{k, \tau} \varphi$ as $\tau\to 0$. In addition, we derive estimates for the differences $| I_{ j } ( \tau_{ 1 } ) - I_{ j } ( \tau_{ 2 } ) |$ for arbitrary positive parameters $\tau_{ 1 }, \tau_{ 2 } < \tilde{ \tau }$ and for each $j = 1, 2, 3$.

Proceeding as in the previous derivation, we obtain
\begin{align}\label{20}
    | I_{ 1 } ( \tau_{ 1 } ) - I_{ 1 } ( \tau_{ 2 } ) | & = \Big| \int_0^{\infty} \Big( \frac{1 - \chi^2(  r^{ 2 \alpha - 2 s } - k^{ 2 \alpha - 2 s } ) }{ r^{ 2 \alpha } - k^{ 2 \alpha } - \mathrm{i} \tau_{1} } - \frac{1 - \chi^2(  r^{ 2 \alpha - 2 s } - k^{ 2 \alpha - 2 s } ) }{ r^{ 2 \alpha } - k^{ 2 \alpha } - \mathrm{i} \tau_{2} } \Big) d r 
    \notag \\ & \quad \times \int_{ | \xi | = r } \hat{\varphi} ( \xi )  \overline{ \hat{ \psi } } ( \xi ) d S( \xi ) \Big| 
    \notag \\ & \leq \Big| \int_0^{\infty} \frac{ \mathrm{i} ( \tau_{1} - \tau_{2}  ) ( 1 - \chi^2(  r^{ 2 \alpha - 2 s } - k^{ 2 \alpha - 2 s } ) ) }{ ( r^{ 2 \alpha } - k^{ 2 \alpha } - \mathrm{i} \tau_{1} ) ( r^{ 2 \alpha } - k^{ 2 \alpha } - \mathrm{i} \tau_{2} ) } d r  \int_{|\xi|=r} \hat{\varphi}(\xi) \overline{ \hat{ \psi } } (\xi) d S(\xi) \Big|
    \notag \\ & \leq \int_0^{\infty} \frac{ | \tau_{1} - \tau_{2}  | ( 1 - \chi^2(  r^{ 2 \alpha - 2 s } - k^{ 2 \alpha - 2 s } ) ) }{ ( r^{ 2 \alpha } - k^{ 2 \alpha} )^{ 2 } - \tau_{ 1 } \tau_{ 2 } } d r \int_{|\xi|=r}| \hat{\varphi}(\xi)| |\hat{\psi}(\xi)| d S(\xi)\notag \\ 
    & \leq C_{p} \int_0^{\infty} \frac{ | \tau_{1} - \tau_{2} | ( 1 - \chi^2(  r^{ 2 \alpha - 2 s } - k^{ 2 \alpha - 2 s } ) ) }{ ( r^{ 4s } + 1 )^{ \frac{ 1 }{ p } } ( k^{ 4s } - 1 - \tau_{ 1 } \tau_{ 2 } )^{ \frac{ 1 }{ q } } } d r \int_{|\xi|=r}| \hat{\varphi}(\xi)| |\hat{\psi}(\xi)| d S(\xi) \notag \\ 
    & \leq \tilde{\tau} C_p k^{ - 4s ( 1 - \frac{1}{p} ) }\|\varphi\|_{H_\delta^{ - 2 s / p }(\mathbb{R}^{ d })}\|\psi\|_{H_\delta^{ - 2 s / p }(\mathbb{R}^{ d })}
    \notag \\ & \leq \tilde{ \tau } C_{ p } k^{ - 4s ( 1 - \frac{ 1 }{ p } ) } \| \varphi \|_{ H_{ \delta }^{ - s / p } ( \mathbb{R}^{ d } ) } \| \psi \|_{ H_{ \delta }^{ - s / p } ( \mathbb{R}^{ d } ) }.
\end{align}
Similarly,
\begin{align}\label{21}
    | I_{ 2 } ( \tau_{ 1 } ) - I_{ 2 } ( \tau_{ 2 } ) | & = \Big| \int_{\mathbb{S}^{ d - 1 } } \langle k\rangle^{ - \frac{ s }{ p } } \hat{\varphi}(k \omega) \langle k\rangle^{ - \frac{ s }{ p } } \overline{ \hat{ \psi } } ( k \omega ) 
    \notag \\ & \quad \times \int_0^{\infty} \frac{ \mathrm{i} ( \tau_{1} - \tau_{2} ) \langle r\rangle^{ \frac{ 2s }{ p } } r^{ d - 1 } \chi^2( r^{ 2 \alpha - 2 s } - k^{ 2 \alpha - 2 s } ) }{ ( r^{ 2 \alpha } - k^{ 2 \alpha } - \mathrm{i} \tau_{1} ) ( r^{ 2 \alpha } - k^{ 2 \alpha } -\mathrm{i} \tau_{2} ) } d r d S(\omega) \Big|
    \notag \\ & \leq \int_{|\xi|=k} \langle \xi \rangle^{ - \frac{ s }{ p } } |\hat{\varphi}(\xi)| \langle\xi\rangle^{ - \frac{ s }{ p } } |\hat{\psi}(\xi)| 
    \notag \\ & \quad \times \int_{\{ r \in \mathbb{R}_{+} : \tau_{ 0 } \leq | r^{ 2 \alpha - 2 s } - k^{ 2 \alpha - 2 s } | \leq 2 \} } \frac{ | \tau_{1} - \tau_{2} | \langle r\rangle^{ \frac{ 2s }{ p } } }{ \tau_0^{2} k^{ 4s } } \Big( \frac{ r }{ k } \Big)^{ d - 1 } d r d S(\xi)
    \notag \\ & \quad + \int_{|\xi|=k} \langle \xi \rangle^{ - \frac{ s }{ p } } |\hat{\varphi}(\xi)| \langle\xi\rangle^{ - \frac{ s }{ p } } | \hat{\psi}(\xi) | \int_{ \Gamma_{ k, \tau_0 } } \frac{ | \tau_{1} - \tau_{2} | ( 1 + |r|^{2} )^{ \frac{ s }{ p } } }{ \tau_0^{2} k^{ 4s } }\Big( \frac{ |r| }{ k } \Big)^{ d - 1 } d r d S(\xi)
    \notag \\ & \leq C_{ \alpha, s, p, \tau_0, d } \int_{|\xi|=k}  \langle \xi \rangle^{ - \frac{ s }{ p } } | \hat{\varphi}(\xi) | \langle \xi \rangle^{ - \frac{ s }{ p } } | \hat{\psi}(\xi) | 
    \notag \\ & \quad \times \int_{ \Gamma_{k, \tau_0} \cup \{r \in \mathbb{R}_{+} : \tau_{ 0 } \leq | r^{ 2 \alpha - 2 s } - k^{ 2 \alpha - 2 s } | \leq 2 \} } \frac{ | \tau_{1} - \tau_{2} | \langle k\rangle^{ \frac{ 2s }{ p } }}{ \tau_0^{2} k^{ 4s } } d r d S(\xi) 
    \notag \\ & \leq \tilde{\tau} C_{ \alpha, s, p, \tau_0, d, \epsilon} k^{ - 4s ( 1 - \frac{1}{2p} ) } \|\varphi\|_{H_{ 1 / 2 + \epsilon}^{ - s / p } ( \mathbb{R}^{ d } ) } \|\psi\|_{H_{ 1 / 2 + \epsilon}^{ - s / p } ( \mathbb{R}^{ d } ) }.
\end{align}
By (\ref{15}), we have for any $\gamma \in ( 0, 1 )$ and any $z \in \mathbb{C}$ that 
\begin{align*}
    | z | \geq C_{ \gamma } ( \Re z )^{ \gamma } ( \Im z)^{ 1 - \gamma }. 
\end{align*}
Assuming $\tau_1 \leq \tau_2$, we obtain 
\begin{align*}
    \Big| \frac{ 1 }{ r^{ 2 \alpha } - k^{ 2 \alpha } - \mathrm{i} \tau_1 } - \frac{ 1 }{ r^{ 2 \alpha } - k^{ 2 \alpha } - \mathrm{i} \tau_2 } \Big| \leq \frac{ 1 }{ | r^{ 2 \alpha } - k^{ 2 \alpha } | } \frac{ C_{ \gamma } \tau_2 }{ | r^{ 2 \alpha } - k^{ 2 \alpha } |^{ \gamma } \tau_{ 2 }^{ 1 - \gamma } } \leq \frac{ C_{ \gamma } \tau_{ 2 }^{ \gamma } }{ | r^{ 2 \alpha } - k^{ 2 \alpha } |^{ 1 + \gamma } }.
\end{align*}
Consequently, for $\beta \in ( 0, \epsilon )$ and $\gamma \in ( 0, \frac{ 1 + 2 s - 2 \alpha }{ 2 \alpha } )$, we get 
\begin{align}\label{22}
    | I_{ 3 } ( \tau_{ 1 } ) - I_{ 3 } ( \tau_{ 2 } ) | & \leq C_{ \gamma, s, d, \alpha, \beta, \epsilon, p } k^{ - 2s ( 1 -\frac{1}{p} ) } \int_{ 0 }^{ \infty } \frac{ \chi^{ 2 } ( r^{ 2 \alpha - 2 s } - k^{ 2 \alpha - 2 s } ) | r - k |^{ \beta } }{ | r^{ 2 \alpha - 2 s } - k^{ 2 \alpha - 2 s } | } 
    \notag \\ & \quad \times \frac{  \tau_{ 2 }^{ \gamma } }{ | r^{ 2 \alpha } - k^{ 2 \alpha } |^{ \gamma } } d r  \| \varphi \|_{ H^{ - s / p }_{ 1 / 2 + \epsilon }(\mathbb{R}^{ d })} \| \psi \|_{H^{ - s / p }_{ 1 / 2 + \epsilon }(\mathbb{R}^{ d })}
    \notag \\ & \leq \tilde{ \tau }^{ \gamma } C_{ \gamma, s, d, \alpha, \beta, \epsilon, p } k^{ - 2s ( 1 -\frac{1}{p} ) } \| \varphi \|_{ H^{ - s / p }_{ 1 / 2 + \epsilon }( \mathbb{R}^{ d } ) } \| \psi \|_{ H^{ - s / p }_{ 1 / 2 + \epsilon } ( \mathbb{R}^{ d } ) }.
\end{align}

Combining (\ref{20}), (\ref{21}), and (\ref{22}), we have for all $\tau_{ 1 }, \tau_{ 2 } \in ( 0, \tilde{ \tau } )$ that
\begin{align*}
    \| \mathcal{ H }_{ k, \tau_{ 1 } } \varphi - \mathcal{H}_{k, \tau_2} \varphi\|_{H^{  s / p }_{ - 1 / 2 - \epsilon }(\mathbb{R}^{ d })} \lesssim \tilde{ \tau }^{ \gamma } \|\varphi\|_{H^{ - s / p }_{ 1 / 2 + \epsilon } ( \mathbb{R}^{ d } ) }.
\end{align*}
Thus, the family $\mathcal{H}_{k, \tilde{\tau}} \varphi$ is Cauchy in $H^{ s/p }_{ - 1/2 - \epsilon } ( \mathbb{R}^{ d } )$ and therefore convergent. Moreover,
\begin{align}\label{24}
    \lim _{\tilde{\tau} \rightarrow 0^{+}} \mathcal{H}_{k, \tilde{\tau}} \varphi = \mathcal{H}_k \varphi \quad \text { in } \, {H^{ s / p }_{ - 1 / 2 - \epsilon }(\mathbb{R}^{ d })}.
\end{align}
Finally, using (\ref{25}) and (\ref{24}), we conclude that
\begin{align*}
    \| \mathcal{H}_{ k } \varphi \|_{H_{ - 1 / 2 - \epsilon }^{ s / p }(\mathbb{R}^{ d })} \leq C_{ \alpha, s, p, d, \epsilon } k^{ - 2s ( 1 - \frac{1}{p} ) } \|\varphi\|_{H_{ 1 / 2 + \epsilon }^{ - s / p } ( \mathbb{R}^{ d } ) },
\end{align*}
which completes the proof. 
\end{proof}

Define the operator $\mathcal{K}_{ k }$ by 
\begin{align*}
    (\mathcal{K}_k \varphi)(x) : = \int_{\mathbb{R}^{ d } } G^{k} ( x, z )  q ( z ) \varphi ( z ) d z, \quad x \in \mathbb{R}^{ d }.
\end{align*}
This operator exhibits the same decay rate in the wavenumber $k$ as the operator $\mathcal{H}_{ k }$. 

\begin{theorem}\label{112}
Let $s \in [ \frac{ \alpha }{ 2 }, \alpha )$, $\epsilon > 0$, and $p \in ( 1, \infty )$. If $k > 2^{ \frac{1}{ 2 \alpha - 2 s } }$ and $q \in L^{ 2 } ( U )$, then 
\begin{align*}
    \| \mathcal{ K }_{ k } \varphi \|_{ H_{ - 1 / 2 - \epsilon }^{ s / p } ( \mathbb{R}^{ d } ) } \lesssim k^{ - 2 s ( 1 - \frac{ 1 }{ p } ) } \| \varphi \|_{ H_{ - 1 / 2 - \epsilon }^{ s / p }( \mathbb{ R }^{ d } ) }, \quad \forall\, \varphi \in H_{ - 1 / 2 - \epsilon }^{ s / p } ( \mathbb{ R }^{ d } ).
\end{align*}
\end{theorem}

\begin{proof}
By \cite[Theorem 2.2]{Li-Liu-Ma-CMP-21}, for any $\varphi \in H_{ - 1 / 2 - \epsilon }^{ s / p } ( \mathbb{ R }^{ d } )$ and $q \in L^{ 2 } ( U )$, we have
\begin{align}\label{106}
    q \varphi \in H_{ 1 / 2 + \epsilon }^{ - s / p } ( \mathbb{ R }^{ d } ).
\end{align}
It follows from Theorem \ref{6} that 
\begin{align*}
    \| \mathcal{ K }_{ k } \varphi \|_{ H_{ - 1 / 2 - \epsilon }^{ s / p } ( \mathbb{R}^{ d } ) } &= \| \mathcal{ H }_{ k } ( q \varphi ) \|_{ H_{ - 1 / 2 - \epsilon }^{ s / p } ( \mathbb{R}^{ d } ) }\\
    &\lesssim k^{ - 2 s ( 1 - \frac{ 1 }{ p } ) }\|q\varphi\|_{H_{ 1 / 2 + \epsilon }^{ - s / p } ( \mathbb{ R }^{ d } )}\lesssim k^{ - 2 s ( 1 - \frac{ 1 }{ p } ) }\|\varphi\|_{H_{ - 1 / 2 - \epsilon }^{ s / p }(\mathbb R^d)},
\end{align*}
which completes the proof. 
\end{proof}

Using this notation, the stochastic fractional Helmholtz equation \eqref{119} can be rewritten formally in the form of the Lippmann--Schwinger equation:
\begin{align}\label{26}
     ( I + \mathcal{K}_{ k } ) u = \mathcal{H}_{k} f.
\end{align}

\begin{theorem}\label{37}
Assume that $m \in ( d - 2 \alpha, d )$ and $q \in L^{ 2 } ( U )$. For any $t \in ( \frac{ d - m }{ 2 }, \alpha )$ and $\epsilon > 0$, if the wavenumber $k$ is sufficiently large so that $\|\mathcal{ K }_k \|_{ \mathcal{L} ( H_{ - 1 / 2 - \epsilon}^{ t } ( \mathbb{ R }^{ d } ) ) } < 1 $, then there exists a unique stochastic process $u : \mathbb{R}^{ d } \rightarrow \mathbb{C}$ satisfying \eqref{26} almost surely. Moreover, the solution satisfies the estimate
\begin{align}\label{28}
    \| u  \|_{ H_{ - 1 / 2 - \epsilon }^{ t } ( \mathbb{R}^{ d } ) } \lesssim \| f \|_{ H_{ 1 / 2 + \epsilon }^{ - t } ( \mathbb{R}^{ d } ) }.
\end{align}
\end{theorem}

\begin{proof}
For any $t \in ( \frac{ d - m }{ 2 }, \alpha )$, Lemma \ref{63} together with Theorem \ref{6} implies that
\begin{align*}
    f \in H_{ 1 / 2 + \epsilon }^{ - t }( \mathbb{R}^{ d } ), \quad \mathcal{H}_{ k } f \in H_{ - 1 / 2 - \epsilon }^{ t }( \mathbb{R}^{ d } ), \quad \text{a.s.}.
\end{align*}
Since $\| \mathcal{K}_{ k } \|_{ \mathcal{L} ( H_{ - 1 / 2 - \epsilon }^{ t } ( \mathbb{ R }^{ d } ) ) } < 1 $, the operator $ I + \mathcal{K}_k $ is invertible on $ H_{ - 1 /2  - \epsilon }^{ t } ( \mathbb{R}^{ d } )$ by \cite[Theorem 7.3-1]{Kreyszig-91}.

Define $u : = ( I + \mathcal{K}_k)^{-1} ( \mathcal{H}_{k} f )$. Then 
$u$ satisfies \eqref{26} almost surely, establishing existence. Estimate \eqref{28} follows directly from \cite[Theorem 7.3-1]{Kreyszig-91} and Theorem \ref{6}. Uniqueness is an immediate consequence of \eqref{28}.
\end{proof}

The following result establishes the well-posedness of the direct problem \eqref{119}--\eqref{102} in the distributional sense.

\begin{theorem}\label{104}
Assume that $m \in ( d - 2 \alpha, d )$ and $q \in L^{ 2 } ( U ) $. For any $t \in ( \frac{ d - m }{ 2 }, \alpha )$ and $\epsilon > 0$, if the wavenumber $k$ is sufficiently large such that $\|\mathcal{K}_k \|_{ \mathcal{L} ( H_{ - 1 / 2 - \epsilon}^{ t } ( \mathbb{ R }^{ d } ) ) } < 1 $, then the scattering problem \eqref{119}--\eqref{102} admits a unique distributional solution $u \in H_{ - 1 / 2 - \epsilon }^{ t } ( \mathbb{R}^{ d } )$ almost surely.
\end{theorem}

\begin{proof}
We first show that any solution of the Lippmann--Schwinger equation \eqref{26} is also a solution of \eqref{119} in the distributional sense. Let $u^{*}$ be the solution of \eqref{26}, so that
\begin{align}\label{121}
    u^{*} ( x ) + \int_{ \mathbb{R}^{ d } } G^{ k } ( x - y ) q ( y ) u^{*} ( y ) d y = \int_{ \mathbb{R}^{ d } } G^{ k } ( x - y ) f ( y ) d y, \quad x \in \mathbb{R}^{ d }.
\end{align}
Since $G^{ k }$ is the fundamental solution of the operator $( - \Delta )^{ \alpha } - k^{ 2 \alpha }$, we have
\begin{align*}
    ( - \Delta )^{ \alpha } G^{ k } ( x ) - k^{ 2 \alpha } G^{ k } ( x ) = \delta ( x  ).
\end{align*}
Hence, for any $\psi \in \mathscr{D}(\mathbb R^d)$, 
\begin{align}\label{148}
    \big\langle \big( ( - \Delta )^{ \alpha } - k^{ 2 \alpha } \big) G^{ k } ( x - y ), \psi ( x ) \big\rangle  = \langle \delta ( x - y ), \psi ( x ) \rangle = \psi ( y ).
\end{align}
Using (\ref{121}) and \eqref{148}, we get for any test function $\psi \in \mathscr{D}(\mathbb R^d)$, 
\begin{align*}
    & \big \langle ( - \Delta )^{ \alpha } u^{*} ( x ) - k^{ 2 \alpha } u^{*} ( x )  + q ( x ) u^{*} ( x ), \psi ( x ) \big \rangle 
    \\ & = \Big \langle \big( ( - \Delta )^{ \alpha } - k^{ 2 \alpha } \big) \big( - \int_{ \mathbb{R}^{ d } } G^{ k } ( x - y ) q ( y ) u^{*} ( y ) d y + \int_{ \mathbb{R}^{ d } } G^{ k } ( x - y ) f ( y ) d y \big ), \psi ( x ) \Big \rangle
    \\ & \quad + \langle q ( x ) u^{*} ( x ), \psi ( x ) \rangle 
    \\ & = - \int_{ \mathbb{R}^{ d } } \Big \langle \big( ( - \Delta )^{ \alpha } - k^{ 2 \alpha } \big) G^{ k } ( x - y ), \psi ( x ) \Big \rangle q ( y ) u^{*} ( y ) d y 
    \\ &\quad  + \int_{ \mathbb{R}^{ d } } \Big \langle \big( ( - \Delta )^{ \alpha } - k^{ 2 \alpha } \big) G^{ k } ( x - y ), \psi ( x ) \Big \rangle f ( y )   d y + \langle q ( x ) u^{ * } ( x ), \psi ( x ) \rangle 
    \\ & = - \langle q  u^{*} , \psi \rangle + \langle f , \psi \rangle + \langle q  u^{*} , \psi \rangle= \langle f , \psi \rangle,
\end{align*}
which shows that $u^{ * }$ satisfies \eqref{119} in the distributional sense. This establishes the existence of a distributional solution to \eqref{119}.

Next, we show the uniqueness of solutions to \eqref{119}. Let $v^* \in H_{ - 1 / 2 - \epsilon }^{ t } ( \mathbb{R}^{ d } )$ be any distributional solution of \eqref{119} with $f = 0$. It suffices to show that $v^*$ also satisfies the Lippmann--Schwinger equation \eqref{26} with $f = 0$, which then implies $v^* \equiv 0$. 

Since $v^*$ is a distributional solution, it satisfies
\begin{align*}
    ( - \Delta )^{ \alpha } v^* ( x ) - k^{ 2 \alpha } v^* ( x  ) + q ( x ) v^* ( x ) = 0.
\end{align*}
Here $q \in L^2 ( \mathbb{R}^{ d } )$, $v^{ * } \in H_{ - 1 / 2 - \epsilon }^{ t } ( \mathbb{R}^{ d } )$, and thus by (\ref{106}), $qv^{ * } \in H_{ 1 / 2 + \epsilon }^{ - t } ( \mathbb{R}^{ d } )$. By Theorem \ref{6}, the Green operator satisfies $G^{ k } ( x, \cdot ) : H_{ 1 / 2 + \epsilon }^{ - t } ( \mathbb{R}^{ d } ) \rightarrow H_{ - 1 / 2 - \epsilon }^{ t } ( \mathbb{R}^{ d } )$ for $t \in ( 0, \alpha )$. Hence, 
\begin{align*}
    \langle G^{ k } ( x, \cdot ), [ ( - \Delta )^{ \alpha } - k^{ 2 \alpha } ] v^* ( \cdot ) \rangle = - \langle G^{ k } ( x, \cdot ), q ( \cdot ) v^* ( \cdot ) \rangle \in H_{ - 1 / 2 - \epsilon }^{ t } ( \mathbb{R}^{ d } ).
\end{align*}
Since the fractional Laplacian is self-adjoint,
\begin{align*}
    \langle G^{ k } ( x, \cdot ), ( - \Delta )^{ \alpha } v^* ( \cdot ) \rangle = \langle ( - \Delta )^{ \alpha } G^{ k } ( x, \cdot ), v^* ( \cdot ) \rangle.
\end{align*}
Therefore,
\begin{align*}
    - \langle G^{ k } ( x, \cdot ), q ( \cdot ) v^* ( \cdot ) \rangle  = \langle [ ( - \Delta )^{ \alpha } - k^{ 2 \alpha } ] G^{ k } ( x, \cdot ), v^* ( \cdot ) \rangle= \langle \delta ( x - \cdot ), v^* ( \cdot ) \rangle= v^* ( x ).
\end{align*}
Equivalently,
\begin{align*}
    v^{ * } ( x ) + \int_{ \mathbb{ R }^{ d } } G^{ k } ( x, y ) q ( y ) v^{ * } ( y ) d y = 0,
\end{align*}
showing that $v^*$ satisfies the Lippmann--Schwinger equation \eqref{26} with $f = 0$. By Theorem \ref{37}, this implies $v^* \equiv 0$, which completes the proof. 
\end{proof}

\section{The inverse problem}\label{sec4}

This section is devoted to establishing the uniqueness of the inverse problem. First, we introduce the Born series derived from the Lippmann--Schwinger integral equation \eqref{26}. Under this formulation, the Born series is shown to converge to the exact solution. The uniqueness result is then obtained by deriving estimates for the Born expansion and analyzing the far-field patterns of the resulting wave field.

\subsection{The Born series}

When the wavenumber $k$ is sufficiently large so that $ \| \mathcal{ K }_{ k } \|_{ \mathcal{ L } ( H_{ - 1 / 2 - \epsilon }^{ t } ( \mathbb{ R }^{ d } ) ) } < 1 $, it is shown in Theorem \ref{104} that the scattering problem \eqref{119}--\eqref{102} has a unique distributional solution $u$, and moreover this solution satisfies the Lippmann--Schwinger equation \eqref{26}.

Define $u_{ 0 } ( x, k ) = ( \mathcal{ H }_{ k } f ) ( x )$, and introduce the Born sequence associated with \eqref{26} by
\begin{align}\label{105}
    u_{ j } ( x, k ) = ( - \mathcal{ K }_{ k } u_{ j - 1 } ( \cdot, k ) ) ( x ), \quad j \geq 1.
\end{align}
We claim that the Born series \eqref{105} converges in $H_{ - 1 / 2 - \epsilon}^{ t } ( \mathbb{R}^{ d } )$ to the unique solution of \eqref{26} for sufficiently large wavenumber $k$, i.e.,
\begin{align}\label{111}
    u ( x, k ) = \sum_{ j \geq 0 } u_{ j } ( x, k ) = \sum_{ j \geq 0 } \big( ( - \mathcal{ K }_{ k } )^{ j } \mathcal{ H }_{ k } f \big) ( x ).
\end{align}

By Theorem \ref{6} and Theorem \ref{112}, the convergence of the Born series follows from the estimate
\begin{align*}
    \Big\| \sum_{ j = N_{ 1 } }^{ N_{ 2 } } u_{ j } ( \cdot, k ) \Big\|_{ H_{ - 1 / 2 - \epsilon}^{ t } ( \mathbb{R}^{ d } ) } & \leq \sum_{ j = N_{ 1 } }^{ N_{ 2 } } \Big\| ( - \mathcal{ K }_{ k } )^{ j } \mathcal{ H }_{ k } f \Big\|_{ H_{ - 1 / 2 - \epsilon}^{ t } ( \mathbb{R}^{ d } ) }
    \\ & \lesssim \sum_{ j = N_{ 1 } }^{ N_{ 2 } } k^{ - 2 s ( 1 - \frac{ 1 }{ p } ) ( j + 1 ) } \| f \|_{ H_{ 1 / 2 + \epsilon}^{ - t } ( \mathbb{R}^{ d } ) } \rightarrow 0,
\end{align*}
as $N_{ 1 }, N_{ 2 } \rightarrow \infty$, provided that $s \in [ \frac{ \alpha }{ 2 }, \alpha ), \alpha > \frac{ s }{ p } = t > \frac{ d - m }{ 2 }$. Let $u^{*} ( x, k ) : = \sum_{ j \geq 0 } u_{ j } ( x, k )$ denote the limit of the Born series. Then
\begin{align*}
    \mathcal{ K }_{ k } u^{*} = \sum_{ j \geq 0 } \mathcal{ K }_{ k } u_{ j } ( x, k ) = - u^{*} + \mathcal{ H }_{ k } f,
\end{align*}
which shows that $u^*$ satisfies the Lippmann--Schwinger equation \eqref{26}, thereby proving the claim.

Next, we analyze the far-field behavior of the wave field $u$. As $| x | \rightarrow \infty$, the Green's function $G^{ k }$ admits the asymptotic expansions (cf. \eqref{Gka}):
\begin{align*}
    G^{ k } ( x ) = 
    \begin{cases}
        \frac{ \mathrm{i} }{ 2 \alpha } k^{ 1 - 2\alpha } e^{ \mathrm{i} k | x | } + O ( | x |^{ - 1 - 2 \alpha } ), & d = 1,\\[3pt] 
        \frac{ 1 + \mathrm{i} }{ 4 \alpha \sqrt{ \pi } } k^{ \frac{ 3 }{ 2 } - 2 \alpha } | x |^{ - \frac{ 1 }{ 2 } } e^{ \mathrm{ i } k | x | } + O ( | x |^{ - \frac{ 3 }{ 2 } } ), & d = 2,\\[3pt]  
        \frac{ 1 }{ 4 \pi \alpha } k^{ 2 - 2\alpha } | x |^{ - 1 } e^{ \mathrm{ i } k | x | } + O ( | x |^{ - 3 - 2 \alpha } ), & d = 3.
    \end{cases}
\end{align*}
Using the standard expansion
\begin{equation*}
    | x - y | = |x| - \hat{x} \cdot y + O \Big( \frac{1}{|x|} \Big), \quad \ |x| \rightarrow \infty,
\end{equation*}
we have
\begin{align}\label{125}
    G^{k} ( x - y ) = C_{ k, d, \alpha } \frac{ e^{ \mathrm{i} k |x| }}{ | x |^{ \frac{ d - 1 }{ 2 } } } e^{ - \mathrm{i} k \hat{x} \cdot y } + O ( | x |^{ - N_{ d } } ), \quad  \ |x| \rightarrow \infty,
\end{align}
where the decay exponents $N_d$ are defined in \eqref{Nd} and the constants $C_{ k, d, \alpha }$ are given by
\begin{align}\label{Ckda}
    C_{ k, d, \alpha } =
    \begin{cases}
        \frac{ \mathrm{i} k^{ 1 - 2 \alpha } }{ 2 \alpha }, & d = 1,\\[3pt]
        \frac{ ( 1 + \mathrm{i} ) k^{ 3 / 2 - 2 \alpha } }{ 4 \alpha \sqrt{ \pi } }, & d = 2,\\[3pt]
        \frac{ k^{ 2 - 2 \alpha } }{ 4 \pi \alpha }, & d = 3.
    \end{cases}
\end{align}

Substituting the asymptotic expansion \eqref{125} into \eqref{111} yields 
\begin{align}\label{aeu}
    u ( x, k ) & = C_{ k, d, \alpha } \frac{ e^{ \mathrm{i} k |x| }}{ | x |^{ \frac{ d - 1 }{ 2 } } } \bigg( \int_{ D } e^{ - \mathrm{i} k \hat{x} \cdot y } f ( y ) dy -  \int_{ U } e^{ - \mathrm{i} k \hat{x} \cdot y } q ( y ) ( \mathcal{H}_{k} f ) ( y ) dy \notag\\ 
    &\quad - \int_{ U } e^{ - \mathrm{i} k \hat{x} \cdot y } q ( y ) \sum_{ j \geq 2 } ( ( - \mathcal{K}_{k} )^{ j - 1 }  \mathcal{H}_{k} f ) ( y ) dy \bigg) + O ( | x |^{ - N_{ d } } ).
\end{align}
Recall the far-field expansion \eqref{ffp}:
\begin{align*}
    u ( x, k ) = \frac{ e^{ \mathrm{i} k |x| }}{ | x |^{ \frac{ d - 1 }{ 2 } } } u^{ \infty } ( \hat{x}, k ) + O ( | x |^{ - N_{ d } } ).
\end{align*}
By comparing \eqref{ffp} with \eqref{aeu}, the far-field pattern is defined as
\begin{align*}
    u^{ \infty } (\hat{x}, k ) : = \sum_{ j = 0, 1, 2 } F_{ j } ( \hat{x}, k ),
\end{align*}
where 
\begin{align*}
    F_{ 0 } ( \hat{x}, k ) & = C_{ k, d, \alpha } \int_{ D } e^{ - \mathrm{i} k \hat{x} \cdot y } f ( y ) dy,
    \\ F_{ 1 } ( \hat{x}, k ) & = - C_{ k, d, \alpha } \int_{ U } e^{ - \mathrm{i} k \hat{x} \cdot y } q ( y ) ( \mathcal{H}_{k} f ) ( y ) dy,
    \\ F_{ 2 } ( \hat{x}, k ) & = - C_{ k, d, \alpha } \int_{ U } e^{ - \mathrm{i} k \hat{x} \cdot y } q ( y ) \sum_{ j \geq 2 } ( ( - \mathcal{K}_{k} )^{ j - 1 }  \mathcal{H}_{k} f ) ( y ) dy.  
\end{align*}

In the subsequent analysis, we estimate the contributions $F_j$, $j=0, 1, 2$. Combining these estimates will lead to the uniqueness result for the inverse scattering problem.

\subsection{Estimates of $F_{0}$}

The following result shows that the functions $\mu^c$ and $\mu^r$ can be uniquely recovered using only the information contained in $F_0$.

\begin{theorem}\label{110}
Let $f$ satisfy Assumption \ref{117}. Then, for any $\tau \geq 0$, the following identities hold almost surely:
\begin{align}
    \hat{\mu^{c}}  ( \tau \hat{x} ) & = \lim \limits_{ K \rightarrow \infty } \frac{ C_{ d, \alpha }^{ c } }{ K } \int_{ K }^{ 2K } k^{ m + 4 \alpha - d - 1 } \overline{ F_{ 0 } ( \hat{x}, k ) } F_{ 0 } ( \hat{x}, k + \tau ) dk, \label{103}\\ 
    \hat{ \mu^{ r } } ( \tau \hat{x} ) & = \lim \limits_{ K \rightarrow \infty } \frac{ C_{ d, \alpha }^{ r } }{ K } \int_{ K }^{ 2K }  k^{ m + 4 \alpha - d - 1 } F_{ 0 } ( - \hat{x}, k ) F_{ 0 } ( \hat{x}, k + \tau ) dk,\label{135}
\end{align}
where
\begin{align*}
    C_{ d, \alpha }^{ c } = 
    \begin{cases}
        4 \alpha^{ 2 }, & d = 1, 
        \\ 8 \pi \alpha^{ 2 }, & d = 2, 
        \\ 16 \pi^{ 2 } \alpha^{ 2 }, & d = 3,
    \end{cases}
    \quad C_{ d, \alpha }^{ r } = 
    \begin{cases}
        - 4 \alpha^{ 2 }, & d = 1, 
        \\ - 8 \pi \alpha^{ 2 } \mathrm{i}, & d = 2, 
        \\ 16 \pi^{ 2 } \alpha^{ 2 }, & d = 3.
    \end{cases}
\end{align*}
\end{theorem}

\begin{proof}
From (\ref{47}) and (\ref{46}), we have
\begin{align*}
    & \mathbb{E} [ \overline{ F_{ 0 } ( \hat{x}, k ) } F_{ 0 } ( \hat{x}, k + \tau ) ] \\ 
    & = \overline{ C_{ k, d, \alpha } } C_{ k + \tau, d, \alpha } \mathbb{E} \Big[ \int_{ D } e^{  \mathrm{i} k \hat{x} \cdot y_{ 1 } } \overline{ f ( y_{ 1 } ) } dy_{ 1 } \int_{ D } e^{  - \mathrm{i} ( k + \tau ) \hat{x} \cdot y_{ 2 } } f ( y_{ 2 } ) dy_{ 2 } \Big]\\
    & = \overline{ C_{ k, d, \alpha } } C_{ k + \tau, d, \alpha } \int_{ D \times D } e^{ \mathrm{i} ( k + \tau ) \hat{x} \cdot ( y_{ 1 } - y_{ 2 } ) } e^{  - \mathrm{i} \tau \hat{x} \cdot y_{ 1 } } \mathbb{E} [ \overline{ f ( y_{ 1 } ) } f ( y_{ 2 } ) ] dy_{ 1 } dy_{ 2 }\\ 
    & = \overline{ C_{ k, d, \alpha } } C_{ k + \tau, d, \alpha } \int_{ D \times D } e^{ \mathrm{i} ( k + \tau ) \hat{x} \cdot ( y_{ 1 } - y_{ 2 } ) } e^{  - \mathrm{i} \tau \hat{x} \cdot y_{ 1 } } K^{ c }_{ f } ( y_{ 1 }, y_{ 2 } ) dy_{ 1 } dy_{ 2 }\\
    & = \overline{ C_{ k, d, \alpha } } C_{ k + \tau, d, \alpha } \int_{ D } e^{  - \mathrm{i} \tau \hat{x} \cdot y_{ 1 } } \Big( \int_{ D } e^{ - \mathrm{i} ( - k - \tau ) \hat{x} \cdot ( y_{ 1 } - y_{ 2 } ) }  K^{ c }_{ f } ( y_{ 1 }, y_{ 2 } ) dy_{ 2 } \Big) dy_{ 1 }\\
    & = \overline{ C_{ k, d, \alpha } } C_{ k + \tau, d, \alpha } \int_{ D } e^{  - \mathrm{i} \tau \hat{x} \cdot y_{ 1 } } \sigma^{ c }_{ f } ( y_{ 1 }, - ( k + \tau ) \hat{x} ) dy_{ 1 } \\ 
    & = \overline{ C_{ k, d, \alpha } } C_{ k + \tau, d, \alpha } \int_{ D } e^{  - \mathrm{i} \tau \hat{x} \cdot y_{ 1 } } \mu^{ c } ( y_{ 1 } ) | ( k + \tau ) \hat{x} |^{ - m } dy_{ 1 } + \overline{ C_{ k, d, \alpha } } C_{ k + \tau, d, \alpha } O ( k^{ - m - 1 } )\\ 
    & =  C^{(1)}_{ d, \alpha }  k^{ d - 4 \alpha - m + 1 } \hat{ \mu^{ c } } ( \tau \hat{x} ) + O ( k^{ d - 4 \alpha - m } ),
\end{align*}
where
\begin{align*}
    C^{ ( 1 ) }_{ d, \alpha }  = 
    \begin{cases}
        \frac{ 1 }{ 4 \alpha^{ 2 } }, & d = 1, \\[3pt] 
        \frac{ 1 }{ 8 \pi \alpha^{ 2 } }, & d = 2, \\[3pt]
        \frac{ 1 }{ 16 \pi^{ 2 } \alpha^{ 2 } }, & d = 3.
    \end{cases}
\end{align*}
Similarly,
\begin{align*}
   & \mathbb{E} [ F_{ 0 } ( - \hat{x}, k ) F_{ 0 } ( \hat{x}, k + \tau ) ] \\
   &=  C_{ k, d, \alpha } C_{ k + \tau, d, \alpha } \int_{ D } e^{  - \mathrm{i} \tau \hat{x} \cdot y_{ 1 } } \Big( \int_{ D } e^{ - \mathrm{i} ( - k - \tau ) \hat{x} \cdot ( y_{ 1 } - y_{ 2 } ) }  K^{ r }_{ f } ( y_{ 1 }, y_{ 2 } ) dy_{ 2 } \Big) dy_{ 1 }\\ 
   & =  C_{ k, d, \alpha } C_{ k + \tau, d, \alpha } \int_{ D } e^{  - \mathrm{i} \tau  \hat{x} \cdot y_{ 1 } } \sigma^{ r }_{ f } ( y_{ 1 }, - ( k + \tau ) \hat{x} ) dy_{ 1 } \\ 
   & = C^{(2)}_{ d, \alpha }  k^{ d - 4 \alpha - m + 1 } \hat{ \mu^{ r } } ( \tau \hat{x} ) + O ( k^{ d - 4 \alpha - m } ),
\end{align*}
where
\begin{align*}
  C^{(2)}_{ d, \alpha } = 
    \begin{cases}
        - \frac{ 1 }{ 4 \alpha^{ 2 } }, & d = 1, \\[3pt] 
        \frac{ \mathrm{i} }{ 8 \pi \alpha^{ 2 } }, & d = 2, \\[3pt] 
        \frac{ 1 }{ 16 \pi^{ 2 } \alpha^{ 2 } }, & d = 3.
    \end{cases}
\end{align*}

To prove \eqref{103}, observe that
\begin{align*}
    \mathbb{E} [ \overline{ F_{ 0 } ( \hat{x}, k ) } F_{ 0 } ( \hat{x}, k + \tau ) ] =  \overline{ C_{ k, d, \alpha } } C_{ k + \tau, d, \alpha } \int_{ \mathbb{R}^{ d } } e^{  - \mathrm{i} \tau \hat{x} \cdot y } \chi ( y ) \sigma^{ c }_{ f } ( y, - ( k + \tau ) \hat{x} ) dy, 
\end{align*}
where $\chi \in C_{ 0 }^{ \infty } ( \mathbb{R}^{ d } )$ satisfies $\chi ( y ) = 1$ for $y \in D$. By the non-stationary phase principle \cite{Stein-93}, for any integer $N \in \mathbb{N}$,
\begin{align}\label{107}
    \mathbb{E} [ \overline{ F_{ 0 } ( \hat{x}, k ) } F_{ 0 } ( \hat{x}, k + \tau ) ] \lesssim ( 1 + k )^{ d - 4 \alpha - m + 1 } ( 1 + \tau )^{ - N }.
\end{align}
Similarly,  
\begin{align*}
    \mathbb{E} [  F_{ 0 } ( \hat{x}, k ) F_{ 0 } ( \hat{x}, k + \tau ) ] = C_{ k, d, \alpha } C_{ k + \tau, d, \alpha } \int_{ \mathbb{R}^{ d } } e^{ - \mathrm{i} ( 2k + \tau ) \hat{x} \cdot y } \chi ( y ) \sigma^{ r }_{ f } ( y, - ( k + \tau ) \hat{x} ) dy.
\end{align*}
Applying again the non-stationary phase principle \cite{Stein-93} yields, for any $N \in \mathbb{N}$,
\begin{align}\label{108}
    \mathbb{E} [ F_{ 0 } ( \hat{x}, k ) F_{ 0 } ( \hat{x}, k + \tau ) ] &\lesssim ( 1 + k )^{ d - 4 \alpha - m + 1 } ( 1 + 2k + \tau )^{ - N }\notag\\
    &\lesssim ( 1 + k )^{ d - 4 \alpha - m + 1 } ( 1 + \tau )^{ - N }.
\end{align}

Define 
\begin{align*}
     F_{ 0 } ( \hat{x}, k ) = U_k + \mathrm{i} V_k,
\end{align*}
where $U_k$ and $V_k$ denote the real and imaginary parts, respectively. Then
\begin{align*}
    \overline{ F_{ 0 } ( \hat{x}, k ) } F_{ 0 } ( \hat{x}, k + \tau ) & = ( U_{ k } - \mathrm{i} V_{ k } ) ( U_{ k + \tau } + \mathrm{i} V_{ k + \tau } )
    \\ & = U_{ k } U_{ k + \tau } + V_{ k } V_{ k + \tau } + \mathrm{i} ( U_{ k } V_{ k + \tau } - V_{ k } U_{ k + \tau } )
    \\ & = \frac{ 1 }{ 4 } \Big[ ( U_{ k } + U_{ k + \tau } )^{ 2 } -  ( U_{ k } - U_{ k + \tau } )^{ 2 } +  ( V_{ k } + V_{ k + \tau } )^{ 2 } -  ( V_{ k } - V_{ k + \tau } )^{ 2 }
    \\ &\quad + \mathrm{i} \big( ( U_{ k } + V_{ k + \tau } )^{ 2 } -  ( U_{ k } - V_{ k + \tau } )^{ 2 } - ( V_{ k } + U_{ k + \tau } )^{ 2 } +  ( V_{ k } - U_{ k + \tau } )^{ 2 } \big) \Big].
\end{align*}

Let $W_k$ denote any of the following random variables:
\begin{align*}
    & U_{ k } + U_{ k + \tau }, \quad U_{ k } - U_{ k + \tau }, \quad  V_{ k } + V_{ k + \tau }, \quad V_{ k } - V_{ k + \tau }, 
    \\ & U_{ k } + V_{ k + \tau },\quad  U_{ k } - V_{ k + \tau }, \quad  V_{ k } + U_{ k + \tau },\quad   V_{ k } - U_{ k + \tau }.
\end{align*}
To establish \eqref{103}, it suffices to show that
\begin{align*}
    \lim_{ K \rightarrow \infty } \frac{ 1 }{ K } \int_K^{ 2 K } k^{ m + 4 \alpha - d - 1 } ( W_k^2 - \mathbb{E} W_k^2 ) d k = 0, \quad \text{a.s.}.
\end{align*}
By \cite[Theorem 4.1]{Caro-Helin-Lassas-AA-19}, it is enough to verify that for any $\tilde{ \tau } \geq 0$, there exist positive constants $c$ and $\epsilon$ such that
\begin{align*}
    | \mathbb{E} [ k^{ m + 4 \alpha - d - 1 }( W_k^2 - \mathbb{E} W_k^2 ) ( k + \tilde{ \tau } )^{ m + 4 \alpha - d - 1 }( W_{ k + \tilde{ \tau } }^2 - \mathbb{E} W_{ k + \tilde{ \tau } }^2 ) ] | \leq c ( 1 + \tilde{ \tau } )^{ - \epsilon }.
\end{align*}
By \cite[Lemma 4.2]{Caro-Helin-Lassas-AA-19}, this in turn reduces to proving
\begin{align*}
    k^{ \frac{ m + 4 \alpha - d - 1 }{ 2 } } ( k + \tilde{ \tau } )^{ \frac{ m + 4 \alpha - d - 1 }{ 2 } } | \mathbb{E} ( W_{ k } W_{ k + \tilde{ \tau } } ) | \leq c ( 1 + \tilde{ \tau } )^{ - \epsilon },
\end{align*}
which follows directly from the decay estimates in \eqref{107} and \eqref{108}.

The identity \eqref{135} follows from an analogous argument, which completes the proof.
\end{proof}

\subsection{Estimates of $F_{1}$}

In this section, we establish estimates for $F_1$ in both two and three dimensions. For the one-dimensional setting, it suffices to analyze $F_0$ and $F_{ 1 } + F_{ 2 }$, and no specific estimates for $F_1$ are needed.

Recall that 
\begin{align}\label{127}
    F_{ 1 } ( \hat{x}, k ) & = - C_{ k, d, \alpha } \int_{ U } e^{ - \mathrm{i} k \hat{x} \cdot y } q ( y ) ( \mathcal{H}_{k} f ) ( y ) dy\notag\\
    & = - C_{ k, d, \alpha } \int_{ U } e^{ - \mathrm{i} k \hat{x} \cdot y } q ( y ) \int_{ D } G^{ k } ( y - z ) f ( z ) dz dy.
\end{align}
To facilitate the forthcoming estimates, we replace $G^{ k }$ by the truncated approximation
\begin{align*}
    G^{ k }_{ N } ( x ) = 
    \begin{cases}
        \frac{ ( 1 + \mathrm{i} ) k^{ 3 / 2 - 2 \alpha } }{ 4 \alpha \sqrt{ \pi | x | } } e^{ \mathrm{ i } k | x | } + C ( 1, x ) k^{ \frac{ 1 }{ 2 } - 2 \alpha } | x |^{ - \frac{ 3 }{ 2 } }, & d = 2,\\[5pt]
        \frac{ k^{ 2 - 2 \alpha } }{ 4 \pi \alpha | x | } e^{ \mathrm{ i } k | x | }, & d = 3,
    \end{cases}
\end{align*}
where
\begin{align*}
    C ( 1, x ) = \frac{ 10 + \mathrm{ i } }{ 32 \alpha } \sqrt{ \frac{ 2 }{ \pi } } \sin ( k | x | - \frac{ \pi }{ 4 } ).
\end{align*}
Define the corresponding truncated far-field contribution
\begin{equation}\label{F1N}
    F_{ 1, N } ( \hat{x}, k ) : = - C_{ k, d, \alpha } \int_{ U } e^{ - \mathrm{i} k \hat{x} \cdot y } q ( y ) \int_{ D } G^{ k }_{ N } ( y - z ) f ( z ) dz dy.
\end{equation}

\begin{lemma}\label{109}
Let $U, D, \alpha$, and $m$ satisfy Assumption \ref{117}. For any $s \in [0,1]$, $p \in (1, \infty)$, and $x \in U$, the following estimates hold:
\begin{align}
    \| G^{ k } ( x, \cdot ) \|_{ W^{ s, p } ( D ) } & \lesssim k^{ \frac{ d }{ 2 } + \frac{ 1 }{ 2 } - 2 \alpha + s }, \label{129}
    \\ \| G^{ k }_{ N } ( x, \cdot ) \|_{ W^{ s, p } ( D ) } & \lesssim  k^{ \frac{ d }{ 2 } + \frac{ 1 }{ 2 } - 2 \alpha + s }, \label{130}
    \\ \| G^{k} ( x, \cdot ) - G^{ k }_{ N } ( x, \cdot ) \|_{ W^{s, p}(D) } & \lesssim 
    \begin{cases}
        k^{ - \frac{ 1 }{ 2 } - 2 \alpha + s}, & d = 2, 
        \\ k^{- 4 \alpha}, & d = 3.
    \end{cases} \label{131}
\end{align}
\end{lemma}

\begin{proof}
Observe that, as $k \rightarrow \infty$, for $x \neq 0$ and $\alpha$ satisfying Assumption \ref{117}, the Green's function admits the asymptotic expansion
\begin{align*}
    G^{ k } ( x ) = 
    \begin{cases}
        \frac{ ( 1 + \mathrm{i} ) k^{ 3 / 2 - 2 \alpha } }{ 4 \alpha \sqrt{ \pi | x | } } e^{ \mathrm{ i } k | x | } + C ( 1, x ) k^{ \frac{ 1 }{ 2 } - 2 \alpha } | x |^{ - \frac{ 3 }{ 2 } } + O ( k^{ - \frac{ 1 }{ 2 } - 2 \alpha } ), & d = 2,\\[5pt] 
        \frac{ k^{ 2 - 2 \alpha } }{ 4 \pi \alpha | x | } e^{ \mathrm{ i } k | x | } + O ( k^{ - 4 \alpha } ), & d = 3. 
    \end{cases}
\end{align*}
Consequently,
\begin{align*}
    G^{ k } ( x ) \lesssim k^{ \frac{ d }{ 2 } + \frac{ 1 }{ 2 } - 2 \alpha }, \quad | \nabla_{ x } G^{ k } ( x ) | \lesssim k^{ \frac{ d }{ 2 } + \frac{ 3 }{ 2 } - 2 \alpha }.
\end{align*}
The estimates (\ref{129}) and (\ref{130}) then follow directly by interpolating between the spaces $L^p(D)$ and $W^{1, p}(D)$.

Comparing $G^k$ with $G^k_N$, we have for $k \rightarrow \infty$ and $x \neq 0$ that 
\begin{align*}
    | G^{k} ( x ) - G^{ k }_{ N } ( x ) | \lesssim 
    \begin{cases}
        k^{ - \frac{ 1 }{ 2 } - 2 \alpha }, & d = 2, 
        \\ k^{ - 4 \alpha }, & d = 3.
    \end{cases}
\end{align*}
A simple calculation yields 
\begin{align*}
    | \nabla_{ x } G^{ k } ( x ) - \nabla_{ x } G^{ k }_{ N } ( x ) | \lesssim 
    \begin{cases}
        k^{ \frac{ 1 }{ 2 } - 2 \alpha }, & d = 2, 
        \\ k^{ - 4 \alpha }, & d = 3.
    \end{cases}
\end{align*}
Interpolating between $L^p(D)$ and $W^{ 1, p } ( D )$ leads directly to the estimate (\ref{131}). 
\end{proof}

\begin{lemma}\label{118}
Let $f$ satisfy Assumption \ref{117}. Then, the following estimate holds almost surely:
\begin{align*}
    | F_{ 1 } ( \widehat{x}, k ) - F_{ 1, N } ( \widehat{x}, k ) | \lesssim k^{ M_{ d } }, 
\end{align*}
where
\begin{align*}
    M_{ d } = 
    \begin{cases}
        2 - 4 \alpha - \frac{ m }{ 2 } + \epsilon, & d = 2, 
        \\ 2 - 6 \alpha, & d = 3.
    \end{cases}
\end{align*}
\end{lemma}

\begin{proof}
It follows from \eqref{127}, \eqref{F1N}, and Lemma \ref{109} that
\begin{align*}
    | F_{ 1 } ( \hat{x}, k ) - F_{ 1, N } ( \hat{x}, k ) | & = | C_{ k, d, \alpha } | \Big| \int_{ U } e^{ - \mathrm{i} k \hat{x} \cdot y } q ( y ) \int_{ D } ( G^{ k } ( y - z ) - G^{ k }_{ N } ( y - z ) ) f ( z ) dz dy  \Big|
    \\ & \lesssim k^{ \frac{ d }{ 2 } + \frac{ 1 }{ 2 } - 2 \alpha } \sup_{ y \in U } \Big| \int_{ D } ( G^{ k } ( y - z ) - G^{ k }_{ N } ( y - z ) ) f ( z ) dz  \Big|
    \\ & \lesssim k^{ \frac{ d }{ 2 } + \frac{ 1 }{ 2 } - 2 \alpha } \| f \|_{ H^{ - s } ( D ) } \sup_{ y \in U } \| G^{ k } ( y - \cdot ) - G^{ k }_{ N } ( y - \cdot ) \|_{ H^{ s } ( D ) }
    \\ & \lesssim k^{ M_{ d } },
\end{align*}
where we choose $s = \frac{ d - m }{ 2 } + \epsilon$.
\end{proof}

The asymptotic behavior of the term $F_1$ in two and three dimensions is characterized in the following lemma. In what follows, we use $\mathcal{C}$ to denote a generic smooth function, whose specific form may vary from line to line.
 
\begin{lemma}\label{101}
Let $f$, $q$ satisfy Assumption \ref{117}, and $\hat{x}$ satisfy $\hat{x} \cdot \hat{n} \geq 0$. Then, in two and three dimensions, the following estimate holds: 
\begin{align*}
    \mathbb{E} \big( | F_{1} ( \hat{x}, k ) |^2 \big) \lesssim
    \begin{cases}
        k^{ \frac{ 7 }{ 2 } - 6 \alpha - m }, & d = 2,
        \\ k^{ 4 - 8 \alpha - \frac{ m }{ 2 } }, & d = 3.
    \end{cases} 
\end{align*}
\end{lemma}

\begin{proof}
First, we derive the estimate in two dimensions. Observe that
\begin{align*}
    & \mathbb{E} \big( | F_{ 1, N } ( \hat{x}, k ) |^2 \big)
    \\ & = | C_{ k, 2, \alpha } |^{ 4 }\, \mathbb{E} \bigg[ \int_{ U } e^{ - \mathrm{i} k \hat{x} \cdot y_{ 1 } } q ( y_{ 1 } ) \int_{ D } | y_{ 1 } - z_{ 1 } |^{ - \frac{ 1 }{ 2 } } e^{ \mathrm{i} k | y_{ 1 } - z_{ 1 } | } + \mathcal{C} ( k, y_{ 1 }, z_{ 1 } ) | y_{ 1 } - z_{ 1 } |^{ - \frac{ 3 }{ 2 } } f ( z_{ 1 } ) dz_{ 1 } dy_{ 1 } \\
    &\quad \times \int_{ U } e^{ \mathrm{i} k \hat{x} \cdot y_{ 2 } } \overline{q ( y_{ 2 } )}  \int_{ D } | y_{ 2 } - z_{ 2 } |^{ - \frac{ 1 }{ 2 } } e^{ - \mathrm{i} k | y_{ 2 } - z_{ 2 } | } + \mathcal{C} ( k, y_{ 2 }, z_{ 2 } ) | y_{ 2 } - z_{ 2 } |^{ - \frac{ 3 }{ 2 } } \overline{ f ( z_{ 2 } ) } dz_{ 2 } dy_{ 2 } \bigg]\\ 
    & \simeq k^{ 6 - 8 \alpha } \int_{ U^{ 2 } \times D^{ 2 } } e^{ \mathrm{i} k ( | y_{ 1 } - z_{ 1 } | - | y_{ 2 } - z_{ 2 } | - \hat{x} \cdot ( y_{ 1 } - y_{ 2 } ) ) } \mathcal{C} ( k, y_{ 1 }, y_{ 2 }, z_{ 1 }, z_{ 2 } ) \\ 
    & \quad \times q ( y_{ 1 } ) \overline{q ( y_{ 2 } )} \mathbb{E} [ f ( z_{ 1 } ) \overline{ f ( z_{ 2 } ) } ] dy_{ 1 } dy_{ 2 } dz_{ 1 } dz_{ 2 } \\ 
    & \simeq k^{ 6 - 8 \alpha } \int_{ U^{ 2 } \times D^{ 2 } } e^{ \mathrm{i} k \varphi ( y_{ 1 }, y_{ 2 }, z_{ 1 }, z_{ 2 } ) } \mathcal{C} ( k, y_{ 1 }, y_{ 2 }, z_{ 1 }, z_{ 2 } ) q ( y_{ 1 } ) \overline{q ( y_{ 2 } )} K^{ c }_{ f } ( z_{ 2 }, z_{ 1 } ) dy_{ 1 } dy_{ 2 } dz_{ 1 } dz_{ 2 } \\
    & \simeq k^{ 6 - 8 \alpha } \int_{ U^{ 2 } \times D^{ 2 } } e^{ \mathrm{i} k \varphi ( y_{ 1 }, y_{ 2 }, z_{ 1 }, z_{ 2 } ) } \mathcal{C} ( k, y_{ 1 }, y_{ 2 }, z_{ 1 }, z_{ 2 } ) q ( y_{ 1 } ) \overline{q ( y_{ 2 } )}\\
    &\quad \times \int_{ \mathbb{R}^{ 2 } } e^{ \mathrm{i} ( z_{2} - z_{1} ) \cdot \xi } \sigma^{ c }_{ f } ( z_{2}, \xi ) d\xi dy_{ 1 } dy_{ 2 } dz_{ 1 } dz_{ 2 }, 
\end{align*}
where 
\begin{align*}
    \mathcal{C} ( k, y_{ 1 }, y_{ 2 }, z_{ 1 }, z_{ 2 } ) = O ( 1 ), \quad  k \rightarrow \infty,
\end{align*}
and
\begin{align*}
    \varphi ( y_{ 1 }, y_{ 2 }, z_{ 1 }, z_{ 2 } ) = - \hat{x} \cdot ( y_{ 1 } - y_{ 2 } ) + | y_{ 1 } - z_{ 1 } | - | y_{ 2 } - z_{ 2 } |.
\end{align*}

To obtain decay with respect to $k$, we act on the oscillatory factor $e^{ \mathrm{i} k \varphi }$ using two specifically designed differential operators. Note that
\begin{align*}
    \nabla_{ y_{ 1 } } \varphi = - \hat{ x } + \frac{ y_{ 1 } - z_{ 1 } }{ | y_{ 1 } - z_{ 1 } | }, \quad \nabla_{ z_{ 1 } } \varphi = \frac{ z_{ 1 } - y_{ 1 } }{ | y_{ 1 } - z_{ 1 } | }.
\end{align*}
We define the following two smooth differential operators:
\begin{align}\label{139}
    L_{ y_{ 1 } } : = \frac{ \nabla_{ y_{ 1 } } \varphi \cdot \nabla_{ y_{ 1 } } }{ \mathrm{i} k | \nabla_{ y_{ 1 } } \varphi |^2 }, \quad L_{ z_{ 1 } } : = \frac{ ( z_{ 1 } - y_{ 1 } ) \cdot \nabla_{ z_{ 1 } } }{ \mathrm{i} k |  y_{ 1 } - z_{ 1 } | }.
\end{align}
For all observation directions $\hat x$ satisfying $\hat x\cdot\hat{n}\geq 0$, the quantity $ | \nabla_{ y_{ 1 } } \varphi |$ admits a strictly positive lower bound. Moreover, by direct computation, we have
\begin{align*}
    L_{ y_{ 1 } } ( e^{ \mathrm{i} k \varphi } ) = e^{ \mathrm{i} k \varphi  }, \quad L_{ z_{ 1 } } ( e^{ \mathrm{i} k \varphi } ) = e^{ \mathrm{i} k \varphi }.
\end{align*}

Denote $z_{ 1 } = ( z_{ 1, 1 }, z_{ 1, 2 } )$. Using the operators $L_{y_1}$, and applying integration by parts, we obtain
\begin{align*}
    \mathbb{E} \big( | F_{ 1, N } ( \hat{x}, k ) |^2 \big) & \simeq k^{ 6 - 8 \alpha } \int_{ U^{ 2 } \times D^{ 2 } } ( L_{ y_{ 1 } }^{ N_{ \alpha } } e^{ \mathrm{i} k \varphi } ) \mathcal{C} ( k, y_{ 1 }, y_{ 2 }, z_{ 1 }, z_{ 2 } ) q ( y_{ 1 } ) \overline{q ( y_{ 2 } )} \\
    &\quad \times \int_{ \mathbb{R}^{ 2 } } e^{ \mathrm{i} ( z_{2} - z_{1} ) \cdot \xi } \sigma_{ f }^{ c } ( z_{2}, \xi ) d\xi dy_{ 1 } dy_{ 2 } dz_{ 1 } dz_{ 2 } \\ 
    & \simeq k^{ 6 - 8 \alpha - N_{ \alpha } } \int_{ U^{ 2 } \times D^{ 2 } } e^{ \mathrm{i} k \varphi } \mathcal{J} \sum_{ j = 0 }^{ N_{ \alpha } } \mathcal{C}_{ j } \overline{ q ( y_{ 2 } ) } ( ( k L_{ y_{ 1 } } )^{ j } q ) ( y_{ 1 } ) dy_{ 1 } dy_{ 2 } dz_{ 1 } dz_{ 2 },
\end{align*}
where
\begin{align*}
    \mathcal{J} = \int_{ \mathbb{R}^{ 2 } } e^{ \mathrm{i} ( z_{2} - z_{1} ) \cdot \xi } \sigma^{ c }_{ f } ( z_{2}, \xi ) d\xi.
\end{align*} 
When $z_{ 1 } \neq z_{ 2 }$, by \cite[Lemmas 3.1--3.2]{Li-Liu-Ma-CMP-21}, 
\begin{align}\label{140}
    | \mathcal{ J } | & = \Big| \int_{ \mathbb{R}^{ 2 } } e^{ \mathrm{i} ( z_{2} - z_{1} ) \cdot \xi } \sigma^{ c }_{ f } ( z_{2}, \xi ) d \xi \Big|
    \notag \\ & = | z_{2} - z_{1} |^{ - \frac{ 3 }{ 2 } } \Big| \int_{ \mathbb{R}^{ 2 } } ( - \Delta )_{ \xi }^{ \frac{ 3 }{ 4 } } ( e^{ \mathrm{i} ( z_{2} - z_{1} ) \cdot \xi } ) \sigma^{ c }_{ f } ( z_{2}, \xi ) d \xi \Big| 
    \notag \\ & = | z_{2} - z_{1} |^{ - \frac{ 3 }{ 2 } } \Big| \int_{ \mathbb{R}^{ 2 } } e^{ \mathrm{i} ( z_{2} - z_{1} ) \cdot \xi } ( - \Delta )_{ \xi }^{ \frac{ 3 }{ 4 } } \sigma^{ c }_{ f } ( z_{2}, \xi ) d \xi \Big|
    \notag \\ & \lesssim | z_{2} - z_{1} |^{ - \frac{ 3 }{ 2 } } \int_{ \mathbb{R}^{ 2 } } \langle \xi \rangle^{ - m - \frac{ 3 }{ 2 } } d \xi \lesssim | z_{ 2 } - z_{ 1 } |^{ - \frac{ 3 }{ 2 } }.
\end{align}
Using \eqref{140} and $q \in H^{ N_{ \alpha } }_{ 0 } ( U )$, we deduce
\begin{align*}
    \mathbb{E} \big( | F_{ 1, N } ( \widehat{x}, k ) |^2 \big) & \lesssim k^{ 6 - 8 \alpha - N_{ \alpha } } \int_{ U^{ 2 } \times D^{ 2 } } | z_{ 1 } - z_{ 2 } |^{ - \frac{ 3 }{ 2 } } \sum_{ j = 0 }^{ N_{ \alpha } } \big| \overline{ q ( y_{ 2 } ) } ( k L_{ y_{ 1 } } )^{ j } q ( y_{ 1 } ) \big| dy_{ 1 } dy_{ 2 } dz_{ 1 } dz_{ 2 } 
    \\ & \lesssim k^{ 6 - 8 \alpha - N_{ \alpha } }.
\end{align*}
By Lemma \ref{118}, since $N_{ \alpha } = \lceil \frac{ 5 }{ 2 } + m - 2 \alpha \rceil$, $m > \frac{ 5 }{ 2 } - 2 \alpha $, and $\alpha > \frac{1}{4}$ in $\mathbb{R}^{2}$, we get
\begin{align}\label{142}
    \mathbb{E} \big( | F_{1} ( \hat{x}, k ) |^2 \big) & \lesssim \mathbb{E} \big( | F_{ 1, N } ( \hat{x}, k ) |^2 \big) + \mathbb{E} \big( | F_{1} ( \hat{x}, k ) - F_{ 1, N } ( \hat{x}, k ) |^2 \big) 
    \notag \\ & \lesssim k^{ 6 - 8 \alpha - N_{ \alpha } } + k^{ 4 - 8 \alpha - m + \epsilon } \lesssim k^{ \frac{ 7 }{ 2 } - 6 \alpha - m }, \quad k \to \infty.
\end{align}

Next, we consider the case in three dimensions. Note that
\begin{align*}
     \mathbb{E} \big( | F_{ 1, N } (\hat{x}, k ) |^2 \big) & = | C_{ k, 3, \alpha } |^{ 4 }\, \mathbb{E} \Big[ \int_{ U } e^{ - \mathrm{i} k \hat{x} \cdot y_{ 1 } } q ( y_{ 1 } ) \int_{ D } | y_{ 1 } - z_{ 1 } |^{ - 1 } e^{ \mathrm{i} k | y_{ 1 } - z_{ 1 } | } f ( z_{ 1 } ) dz_{ 1 } dy_{ 1 } \\
     & \quad \times \int_{ U } e^{ \mathrm{i} k \hat{x} \cdot y_{ 2 } } \overline{q( y_{ 2 } )} \int_{ D } | y_{ 2 } - z_{ 2 } |^{ - 1 } e^{ - \mathrm{i} k | y_{ 2 } - z_{ 2 } | } \overline{ f ( z_{ 2 } ) } dz_{ 2 } dy_{ 2 } \Big]\\ 
     & \simeq k^{ 8 - 8 \alpha } \int_{ U^{ 2 } \times D^{ 2 } } e^{ \mathrm{i} k ( | y_{ 1 } - z_{ 1 } | - | y_{ 2 } - z_{ 2 } | - \hat{x} \cdot ( y_{ 1 } - y_{ 2 } ) ) } \mathcal{C} ( y_{ 1 }, y_{ 2 }, z_{ 1 }, z_{ 2 } ) \\
     &\quad \times q ( y_{ 1 } )\overline{ q ( y_{ 2 } )} \mathbb{E} [ f ( z_{ 1 } ) \overline{ f ( z_{ 2 } ) } ] dy_{ 1 } dy_{ 2 } dz_{ 1 } dz_{ 2 } \\ 
     & \simeq k^{ 8 - 8 \alpha } \int_{ U^{ 2 } \times D^{ 2 } } e^{ \mathrm{i} k ( | y_{ 1 } - z_{ 1 } | - | y_{ 2 } - z_{ 2 } | - \hat{x} \cdot ( y_{ 1 } - y_{ 2 } ) ) } \mathcal{C} ( y_{ 1 }, y_{ 2 }, z_{ 1 }, z_{ 2 } )\\
     &\quad \times q ( y_{ 1 } ) \overline{q ( y_{ 2 } )} K^{ c }_{ f } ( z_{ 2 }, z_{ 1 } ) dy_{ 1 } dy_{ 2 } dz_{ 1 } dz_{ 2 } \\
     & \simeq k^{ 8 - 8 \alpha } \int_{ U^{ 2 } \times D^{ 2 } } e^{ \mathrm{i} k \varphi ( y_{ 1 }, y_{ 2 }, z_{ 1 }, z_{ 2 } ) } \mathcal{C} ( y_{ 1 }, y_{ 2 }, z_{ 1 }, z_{ 2 } ) q ( y_{ 1 } ) \overline{q ( y_{ 2 } )}\\
     &\quad \times \int_{ \mathbb{R}^{ 3 } } e^{ \mathrm{i} ( z_{2} - z_{1} ) \cdot \xi } \sigma^{ c }_{ f } ( z_{2}, \xi ) d\xi dy_{ 1 } dy_{ 2 } dz_{ 1 } dz_{ 2 }, 
\end{align*}
where
\begin{equation*}
    \varphi ( y_{ 1 }, y_{ 2 }, z_{ 1 }, z_{ 2 } ) = - \hat{x} \cdot ( y_{ 1 } - y_{ 2 } ) + | y_{ 1 } - z_{ 1 } | - | y_{ 2 } - z_{ 2 } |.
\end{equation*}

Denote $z_{ 1 } = ( z_{ 1, 1 }, z_{ 1, 2 }, z_{ 1, 3 } )$. Applying integration by parts and using the operators defined in \eqref{139}, we get
\begin{align*}
    \mathbb{E} \big( | F_{ 1, N } ( \hat{x}, k ) |^2 \big) & \simeq k^{ 8 - 8 \alpha } \int_{ U^{ 2 } \times D^{ 2 } } ( L_{ y_{ 1 } }^{ N_{ \alpha } } L_{ z_{ 1 } } e^{ \mathrm{i} k \varphi } )  \mathcal{C} ( y_{ 1 }, y_{ 2 }, z_{ 1 }, z_{ 2 } ) q ( y_{ 1 } ) \overline{q ( y_{ 2 } )} \\
    & \quad \times \int_{ \mathbb{R}^{ 3 } } e^{ \mathrm{i} ( z_{2} - z_{1} ) \cdot \xi } \sigma_{ f }^{ c } ( z_{2}, \xi ) d\xi dy_{ 1 } dy_{ 2 } dz_{ 1 } dz_{ 2 }\\ 
    & \simeq k^{ 7 - 8 \alpha - N_{ \alpha } } \int_{ U^{ 2 } \times D^{ 2 } } e^{ \mathrm{i} k \varphi } \Big\{ \mathcal{L}_{ 1 } \sum_{ l = 0 }^{ N_{ \alpha } } \mathcal{C}_{ l } \overline{q ( y_{ 2 } )} ( ( k L_{ y_{ 1 } } )^{ l } q ) ( y_{ 1 } ) \\
    & \quad + \sum_{ j = 1 }^{ 3 } \mathcal{L}_{ 2, j } \sum_{ l = 0 }^{ N_{ \alpha } } \mathcal{C}_{ l, j } \overline{q ( y_{ 2 } )} ( ( k L_{ y_{ 1 } } )^{ l } q ) ( y_{ 1 } ) \Big\} dy_{ 1 } dy_{ 2 } dz_{ 1 } dz_{ 2 },
\end{align*}
where
\begin{align*}
    \mathcal{L}_{ 1 } = \int_{ \mathbb{R}^{ 3 } } e^{ \mathrm{i} ( z_{2} - z_{1} ) \cdot \xi } \sigma^{ c }_{ f } ( z_{2}, \xi ) d\xi, \quad \mathcal{L}_{ 2, j } = \partial_{ z_{ 1, j } } \int_{ \mathbb{R}^{ 3 } } e^{ \mathrm{i} ( z_{2} - z_{1} ) \cdot \xi } \sigma^{ c }_{ f } ( z_{2}, \xi ) d\xi.
\end{align*}

For $z_{ 1 } \neq z_{ 2 }$, $m \in ( \frac{ 15 }{ 4 } - 2 \alpha, 3 )$, and $\alpha \in ( \frac{ 3 }{ 8 }, 1 )$, applying repeated integration by parts in $\xi$, as used in (\ref{140}), yields
\begin{align*}
    | \mathcal{L}_{ 1 } | \lesssim | z_{ 1 } - z_{ 2 } |^{ - \frac{ 3 }{ 2 } }, \quad | \mathcal{L}_{ 2, k } | \lesssim | z_{ 1 } - z_{ 2 } |^{ - \frac{ 5 }{ 2 } }.
\end{align*}
Thus, for $q \in H^{ N_{ \alpha } }_{ 0 } ( U )$, 
\begin{align*}
    \mathbb{E} \big( | F_{ 1, N } ( \hat{x}, k ) |^{ 2 } \big) &\lesssim k^{ 7 - 8 \alpha - N_{ \alpha } } \int_{ U^{ 2 } \times D^{ 2 } } | z_{ 1 } - z_{ 2 } |^{ - \frac{ 5 }{ 2 } } \sum_{ l = 0 }^{ N_{ \alpha } } | \overline{q ( y_{ 2 } )} ( k L_{ y_{ 1 } } )^{ l } q ( y_{ 1 } ) | dy_{ 1 } dy_{ 2 } dz_{ 1 } dz_{ 2 } \\
    &\lesssim k^{ 7 - 8 \alpha - N_{ \alpha } }.
\end{align*}
Using Lemma \ref{118}, since $N_{ \alpha } = \lceil 3 + \frac{ m }{ 2 } \rceil$, $m \in ( \frac{ 15 }{ 4 } - 2 \alpha, 3 )$, and $\alpha \in ( \frac{ 3 }{ 8 }, 1 )$ in $\mathbb{ R }^{ 3 }$,
\begin{align}\label{143}
    \mathbb{E} \big( | F_{1} ( \hat{x}, k ) |^2 \big) & \lesssim \mathbb{E} \big( | F_{ 1, N } ( \hat{x}, k ) |^2 \big) + \mathbb{E} \big( | F_{1} ( \hat{x}, k ) - F_{ 1, N } ( \hat{x}, k ) |^2 \big) 
    \notag \\ & \lesssim k^{ 7 - 8 \alpha - N_{ \alpha } } + k^{ 4 - 12 \alpha }\lesssim k^{ 4 - 8 \alpha - \frac{ m }{ 2 } }.
\end{align}
Combining (\ref{142}) and (\ref{143}), we complete the proof.
\end{proof}

\subsection{Estimates of $F_{ 2 }$}

First, we estimate the term  $F_{ 2 }$ in two and three dimensions. 

\begin{lemma}\label{113}
Let $d = 2$ or $3$. Let $s \in [ \frac{ \alpha }{ 2 }, \alpha )$ and $p \in ( 1, \infty ) $ satisfy $\frac{ s }{ p } > \frac{ d - m }{ 2 }$. Under Assumption \ref{117}, the following estimate holds almost surely:
\begin{align*}
    | F_{2} ( \hat{x}, k ) | \lesssim k^{ \frac{ d }{ 2 } + \frac{ 1 }{ 2 } - 2 \alpha + \frac{ 5s }{ p } - 4s }.
\end{align*}
\end{lemma}

\begin{proof}
By Theorem~\ref{6}, Theorem~\ref{112}, \cite[Theorem 2.2]{Li-Liu-Ma-CMP-21}, and \cite[Lemma 3.4]{Li-Liu-Ma-CMP-21}, we obtain
\begin{align*}
    | F_{2} ( \hat{x}, k ) | & = \Big| C_{ k, d, \alpha } \int_{ U } e^{ - \mathrm{i} k \hat{x} \cdot y } q ( y ) \sum_{ j \geq 2 } ( ( - \mathcal{K}_{k} )^{ j - 1 }  \mathcal{H}_{k} f ) ( y ) dy  \Big|
    \\ & \leq | C_{ k, d, \alpha } | \sum_{ j \geq 2 } \Big| \int_{ U } e^{ - \mathrm{i} k \hat{x} \cdot y } q ( y ) ( ( - \mathcal{K}_{k} )^{ j - 1 }  \mathcal{H}_{k} f ) ( y ) dy  \Big|
    \\ & \lesssim k^{ \frac{ d }{ 2 } + \frac{ 1 }{ 2 } - 2 \alpha } \| e^{ - \mathrm{i} k \hat{x} \cdot ( \cdot ) } \chi_{ U } ( \cdot ) \|_{ H^{ s / p }_{ - 1/2 - \epsilon } ( \mathbb{R}^{ d } ) } \sum_{ j \geq 2 } \| q ( - \mathcal{K}_{k} )^{ j - 1 }  \mathcal{H}_{k} f \|_{ H^{ - s / p }_{ 1/2 + \epsilon } ( \mathbb{R}^{ d } ) }
    \\ & \lesssim k^{ \frac{ d }{ 2 } + \frac{ 1 }{ 2 } - 2 \alpha } k^{ \frac{ s }{ p } } \sum_{ j \geq 2 } \| ( - \mathcal{K}_{k} )^{ j - 1 }  \mathcal{H}_{k} f \|_{ H^{ s / p }_{ - 1 / 2 - \epsilon } ( \mathbb{R}^{ d } ) }
    \\ & \lesssim k^{ \frac{ d }{ 2 } + \frac{ 1 }{ 2 } - 2 \alpha + \frac{ s }{ p } } \sum_{ j \geq 2 } \| \mathcal{ K }_{ k } \|_{ \mathcal{ L } ( H^{ s / p }_{ - 1 / 2 - \epsilon } ( \mathbb{R}^{ d } ) ) }^{ j - 1 } \| \mathcal{ H }_{ k } \|_{ \mathcal{ L } ( H^{ - s / p }_{ 1 / 2 + \epsilon } ( \mathbb{ R }^{ d } ), H^{ s / p }_{ - 1 / 2 - \epsilon } ( \mathbb{ R }^{ d } ) ) } \| f \|_{ H^{ - s / p }_{ 1 / 2 + \epsilon } ( \mathbb{ R }^{ d } ) }
    \\ & \lesssim k^{ \frac{ d }{ 2 } + \frac{ 1 }{ 2 } - 2 \alpha + \frac{ s }{ p } } \sum_{ j \geq 2 } k^{ - 2 s ( 1 - \frac{ 1 }{ p } ) ( j - 1 ) } k^{ - 2 s ( 1 - \frac{ 1 }{ p } ) } \lesssim k^{ \frac{ d }{ 2 } + \frac{ 1 }{ 2 } - 2 \alpha + \frac{ 5 s }{ p } - 4s }, \quad \text{a.s.},
\end{align*}
where $\chi_{ U } \in C_{ 0 }^{ \infty } ( \mathbb{ R }^{ d } )$ satisfies $\chi_{ U } ( x ) = 1$ for $x \in U$.
\end{proof}

Next, we estimate the quantity $F_{ 1 } + F_{ 2 }$ in the one-dimensional case.

\begin{lemma}\label{144}
Let $d = 1$. Let $s \in [ \frac{ \alpha }{ 2 }, \alpha )$ and $p \in ( 1, \infty )$ satisfy $\frac{ s }{ p } > \frac{ 1 - m }{ 2 }$. Under Assumption \ref{117}, the following estimate holds almost surely:
\begin{align*}
    | F_{ 1 } ( \hat{ x }, k ) + F_{ 2 } ( \hat{ x }, k ) | \lesssim k^{ - 2 \alpha + \frac{ 3 s }{ p } - 2s + 1 }.
\end{align*}
\end{lemma}

\begin{proof}
By Theorem~\ref{6}, Theorem~\ref{112}, \cite[Theorem 2.2]{Li-Liu-Ma-CMP-21}, and \cite[Lemma 3.4]{Li-Liu-Ma-CMP-21}, we have  
\begin{align*}
    | F_{ 1 } ( \hat{x}, k ) + F_{2} ( \hat{x}, k ) | & \leq | C_{ k, 1, \alpha } | \Big| \int_{ U } e^{ - \mathrm{ i } k \hat{x} \cdot y } q ( y ) \sum_{ j \geq 1 } ( ( - \mathcal{K}_{ k } )^{ j - 1 }  \mathcal{ H }_{k} f ) ( y ) dy  \Big|
    \\ & \lesssim k^{ 1 - 2 \alpha } \sum_{ j \geq 1 } \Big| \int_{ U } e^{ - \mathrm{ i } k \hat{x} \cdot y } q ( y ) ( ( - \mathcal{K}_{k} )^{ j - 1 }  \mathcal{ H }_{k} f ) ( y ) dy  \Big|
    \\ & \lesssim k^{ 1 - 2 \alpha } \| e^{ - \mathrm{ i } k \hat{x} \cdot ( \cdot ) } \chi_{ U } ( \cdot ) \|_{ H^{ s / p }_{ - 1/2 - \epsilon } ( \mathbb{ R } ) } \sum_{ j \geq 1 } \|  q ( - \mathcal{K}_{k} )^{ j - 1 }  \mathcal{ H }_{k} f  \|_{ H^{ - s / p }_{ 1/2 + \epsilon } ( \mathbb{ R } ) }
    \\ & \lesssim k^{ 1 - 2 \alpha } k^{ \frac{ s }{ p } } k^{ - 2 s ( 1 - \frac{ 1 }{ p } ) }\lesssim k^{ - 2 \alpha + \frac{ 3 s }{ p } - 2s + 1 }, \quad \text{a.s.},
\end{align*}
where $\chi_{ U } \in C_{ 0 }^{ \infty } ( \mathbb{ R } )$ satisfies $\chi_{ U } ( x ) = 1$ for $x \in U$.
\end{proof}

\subsection{Uniqueness}

The uniqueness of the inverse scattering problem is established in the following theorem, which is the main theoretical result of this study.

\begin{theorem}
Suppose that $f$ and $q$ satisfy Assumption \ref{117}. Then the data set
\begin{align*}
    \mathcal{M}_{f} ( \omega ) = \{ u^{ \infty } ( \hat{x}, k ) : \forall\, \hat{x} \in \mathbb{S}^{ d - 1 },\, k>0  \}
\end{align*}
uniquely determines $\mu^{ c }$ and $\mu^{ r }$ almost surely for each fixed $\omega \in \Omega$. Moreover, they can be recovered via
\begin{align}\label{120}
    \hat{ \mu^{ c } } ( \tau \hat{x} ) =  
    \begin{cases}
        \lim \limits_{ K \to \infty } \frac{ C_{ d, \alpha }^{ c } }{ K } \int_{ K }^{ 2K }  k^{ m + 4 \alpha - d - 1 } \overline{ u^{ \infty } ( \hat{x}, k ) } u^{ \infty } ( \hat{x}, k + \tau ) dk, & \hat{x} \cdot \hat{ n } \geq 0,\\[5pt]
        \overline{ \hat{ \mu^{ c } } ( - \tau \hat{x} ) }, & \hat{x} \cdot \hat{ n } < 0,
    \end{cases}
\end{align}
and
\begin{align}\label{136}
    \hat{ \mu^{ r } } ( \tau \hat{x} ) =  
    \begin{cases}
        \lim \limits_{ K \to \infty } \frac{ C_{ d, \alpha }^{ r } }{ K } \int_{ K }^{ 2K }  k^{ m + 4 \alpha - d - 1 } u^{ \infty } ( - \hat{x}, k ) u^{ \infty } ( \hat{x}, k + \tau ) dk, & \hat{x} \cdot \hat{ n } \geq 0,\\[5pt] \overline{ \hat{ \mu^{ r } } ( - \tau \hat{x} ) }, & \hat{x} \cdot \hat{ n } < 0,
    \end{cases}
\end{align}
for all $\tau \geq 0$, where $u^{ \infty } ( \hat{x}, k ) \in \mathcal{M}_{ f }$, and
\begin{align*}
    C_{ d,\alpha }^{ c } = 
    \begin{cases}
        4 \alpha^{ 2 }, & d = 1, 
        \\ 8 \pi \alpha^{ 2 }, & d = 2, 
        \\ 16 \pi^{ 2 } \alpha^{ 2 }, & d = 3,
    \end{cases}
    \quad C_{ d, \alpha }^{ r } = 
    \begin{cases}
        - 4 \alpha^{ 2 }, & d = 1, 
        \\ - 8 \pi \alpha^{ 2 } \mathrm{i}, & d = 2, 
        \\ 16 \pi^{ 2 } \alpha^{ 2 }, & d = 3.
    \end{cases}
\end{align*}
\end{theorem}

\begin{proof}
We begin by presenting the proof for the one-dimensional case. To verify \eqref{120}, we examine the data correlation expressed as
\begin{align}\label{145}
    & \frac{ 1 }{ K } \int_{ K }^{ 2K } k^{ m + 4 \alpha - 2 } \overline{ u^{ \infty } ( \hat{x}, k ) } u^{ \infty } ( \hat{x}, k + \tau ) dk\notag \\
    & = \frac{ 1 }{ K } \Big\{ \int_{ K }^{ 2K }  k^{ m + 4 \alpha - 2 } \overline{ F_{ 0 } ( k ) } F_{ 0 } ( k + \tau ) dk + \int_{ K }^{ 2K }  k^{ m + 4 \alpha - 2 } \overline{ F_{ 0 } ( k ) } ( F_{ 1 } ( k + \tau ) + F_{ 2 } ( k + \tau ) ) dk\notag \\ 
    &\quad + \int_{ K }^{ 2K }  k^{ m + 4 \alpha - 2 } ( \overline{ F_{ 1 } ( k ) } + \overline{ F_{ 2 } ( k ) } ) F_{ 0 } ( k + \tau ) dk 
    \notag \\ 
    &\quad + \int_{ K }^{ 2K }  k^{ m + 4 \alpha - 2 } ( \overline{ F_{ 1 } ( k ) } + \overline{ F_{ 2 } ( k ) } ) ( F_{ 1 } ( k + \tau ) + F_{ 2 } ( k + \tau ) ) dk \Big\},
\end{align}
where $F_{ i } ( k ) = F_{ i } ( \hat{x}, k )$, $i = 0, 1, 2$. Applying Theorem \ref{110} yields
\begin{align}\label{146}
    \lim \limits_{ K \rightarrow \infty } \frac{ 1 }{ K } \int_{ K }^{ 2K }  k^{ m + 4 \alpha - 2 } \overline{ F_{ 0 } ( k ) } F_{ 0 } ( k + \tau ) dk = \frac{ 1 }{ C_{ 1, \alpha }^{ c } } \hat{\mu^{ c }} ( \tau \hat{x} ), \quad \text{a.s.}.
\end{align}
For the mixed term, by the Cauchy--Schwarz inequality, we obtain
\begin{align*}
    & \frac{ 1 }{ K }\int_{ K }^{ 2K }  k^{ m + 4 \alpha - 2 } \overline{ F_{ 0 } ( k ) } ( F_{ 1 } ( k + \tau ) + F_{ 2 } ( k + \tau ) ) d k 
    \\ & \leq \Big( \frac{ 1 }{ K } \int_{ K }^{ 2K }  k^{ m + 4 \alpha - 2 } | F_{ 0 } ( k ) |^{ 2 } dk \Big)^{ \frac{ 1 }{ 2 } } \Big( \frac{ 1 }{ K } \int_{ K }^{ 2K }  k^{ m + 4 \alpha - 2 } | F_{ 1 } ( k + \tau ) + F_{ 2 } ( k + \tau ) |^{ 2 } dk \Big)^{ \frac{ 1 }{ 2 } }.
\end{align*}
By Lemma \ref{144}, we get
\begin{align*}
    \frac{ 1 }{ K } \int_{ K }^{ 2K } k^{ m + 4 \alpha - 2 } | F_{ 1 } ( k + \tau ) + F_{ 2 } ( k + \tau ) |^{ 2 } dk & \lesssim \frac{ 1 }{ K } \int_{ K }^{ 2K } k^{ m + 4 \alpha - 2 } k^{ 2 - 4 \alpha + \frac{ 6 s }{ p } - 4s } dk
    \\ & \lesssim K^{ m + \frac{ 6 s }{ p } - 4 s }, \quad \text{a.s.},
\end{align*}
for $s \in [ \frac{ \alpha }{ 2 }, \alpha )$ and $p \in ( 1, \infty )$ satisfying $\frac{ s }{ p } > \frac{ 1 - m }{ 2 }$. Choosing
\begin{align*}
    s = \alpha - \epsilon_{ 1 }, \quad  \frac{ s }{ p } = \frac{ 1 - m }{ 2 } + \epsilon_{ 2 },
\end{align*}
we obtain that as $K \rightarrow \infty$,
\begin{align}\label{147}
    \frac{ 1 }{ K } \int_{ K }^{ 2K } k^{ m + 4 \alpha - 2 } | F_{ 1 } ( k + \tau ) + F_{ 2 } ( k + \tau ) |^{ 2 } dk \lesssim K^{ 3 - 2 m - 4 \alpha + \epsilon } \rightarrow 0, \quad \text{a.s.},
\end{align}
for every $m \in ( \frac{ 3 }{ 2 } - 2 \alpha, 1 )$. Putting \eqref{145}, \eqref{146}, and \eqref{147} together establishes \eqref{120} for $d = 1$. The proof of \eqref{136} in one dimension follows a similar argument.

We now proceed to the proof for two and three dimensions. To prove \eqref{120}, we calculate that
\begin{align*}
    \frac{ 1 }{ K } \int_{ K }^{ 2K } k^{ m + 4 \alpha - d - 1 } \overline{ u^{ \infty } ( \hat{x}, k ) } u^{ \infty } ( \hat{x}, k + \tau ) dk & = \sum_{ i, j = 0 }^{ 2 } \frac{ 1 }{ K } \int_{ K }^{ 2K } k^{ m + 4 \alpha - d - 1 } \overline{ F_i ( \hat{x}, k ) } F_j ( \hat{x}, k + \tau ) dk
    \\ & = \sum_{ i, j = 0 }^{ 2 } I_{ i, j } ( \hat{x}, K, \tau ).
\end{align*}
Theorem \ref{110} yields that
\begin{align}\label{114}
    \lim \limits_{ K \rightarrow \infty } I_{ 0, 0 } ( \hat{x}, K, \tau ) = \frac{ 1 }{ C^{ c }_{ d, \alpha } } \hat{ \mu^{ c } } ( \tau \hat{x} ), \quad \text{a.s.}.
\end{align}
For higher-order terms, using the Cauchy--Schwarz inequality, we obtain
\begin{align*}
    | I_{ i, j } | & = \Big | \frac{ 1 }{  K } \int_{ K }^{ 2K } k^{ m + 4 \alpha - d - 1 } \overline{ F_{ i } ( \hat{x}, k ) } F_{ j } ( \hat{x}, k + \tau ) dk \Big|\\
    & \leq \Big( \frac{ 1 }{ K } \int_{ K }^{ 2K } k^{ m + 4 \alpha - d - 1 } | F_{ i } ( \hat{x}, k ) |^{ 2 } dk \Big)^{ \frac{ 1 }{ 2 } } \Big( \frac{ 1 }{ K } \int_{ K }^{ 2K } k^{ m + 4 \alpha - d - 1 } | F_{ j } ( \hat{x}, k + \tau ) |^{ 2 } dk \Big)^{ \frac{ 1 }{ 2 } }.
\end{align*}
By Lemma \ref{101}, for $d = 2$, 
\begin{align*}
    \int_{ 1 }^{ \infty } k^{ m + 4 \alpha - d - 2 } \mathbb{E} | F_{ 1 } ( \hat{x}, k ) |^{ 2 } dk \lesssim \int_{ 1 }^{ \infty } k^{ m + 4 \alpha - 4 } k^{ \frac{ 7 }{ 2 } - 6 \alpha - m } dk < \infty.
\end{align*}
For $d = 3$,
\begin{align*}
    \int_{ 1 }^{ \infty } k^{ m + 4 \alpha - d - 2 } \mathbb{E} | F_{ 1 } ( \hat{x}, k ) |^{ 2 } dk \lesssim \int_{ 1 }^{ \infty } k^{ m + 4 \alpha - 5 } k^{ 4 - 8 \alpha - \frac{ m }{ 2 } } dk < \infty.
\end{align*}
Hence, by \cite[Lemma 4.1]{Li-Liu-Ma-CMP-21}, we get
\begin{align}\label{115}
    \lim \limits_{ K \rightarrow \infty } \frac{ 1 }{ K } \int_{ K }^{ 2K }  k^{ m + 4 \alpha - d - 1 } | F_{ 1 } ( \hat{x}, k ) |^{ 2 } dk = 0, \quad \text{a.s.}.
\end{align}
It follows from Lemma \ref{113} that 
\begin{align*}
    \frac{ 1 }{ K } \int_{ K }^{ 2K } k^{ m + 4 \alpha - d - 1 } | F_{ 2 } ( \hat{x}, k ) |^{ 2 } dk & \lesssim \frac{ 1 }{ K } \int_{ K }^{ 2K } k^{ m + 4 \alpha - d - 1 } k^{ d + 1 - 4 \alpha + \frac{ 10 s }{ p } - 8 s } d k 
    \\ & \lesssim K^{ \frac{ 10s }{ p } - 8s + m }, \quad \text{a.s.}.
\end{align*}
By choosing 
\begin{align*}
    s = \alpha - \epsilon_{ 1 }, \quad \frac{ s }{ p } = \frac{ d - m }{ 2 } + \epsilon_{ 2 },
\end{align*}
we obtain that as $K \rightarrow \infty$, 
\begin{align}\label{116}
    \frac{ 1 }{ K } \int_{ K }^{ 2K } k^{ m + 4 \alpha - d - 1 } | F_{ 2 } ( \hat{x}, k ) |^{ 2 } dk \lesssim K^{ 5 d - 8 \alpha - 4 m + \epsilon } \rightarrow 0, \quad \text{a.s.},
\end{align}
for $m \in ( \frac{ 5 d }{ 4 } - 2 \alpha, d )$ and $\alpha \in ( \frac{ d }{ 8 }, 1 )$. Combining \eqref{114}, \eqref{115}, and \eqref{116}, we obtain \eqref{120} for $d = 2, 3$.  The identity \eqref{136} is proved in the same manner, which completes the proof.
\end{proof}

\section{Conclusion}\label{sec5}

In this work, we investigated the inverse random source problem for the fractional Helmholtz equation. The random source is modeled as a GMIG random field whose covariance and relation operators are classical pseudo-differential operators. We analyzed both the direct and inverse problems. For the direct problem, we established well-posedness by proving the existence and uniqueness of distributional solutions for sufficiently large wavenumbers. For the inverse problem, by combining the Born approximation with tools from microlocal analysis, we derived an explicit connection between the principal symbols of the covariance and relation operators of the random source and the far-field pattern generated by a single realization. This connection enabled us to prove the uniqueness of recovering the principal symbols of both operators from a single measurement.

In the present paper, the potential term is assumed to be deterministic. A challenging and important direction for future research is to develop the direct and inverse scattering theory for fractional Helmholtz equations with random potentials. We hope to report progress on this problem in future work.

\appendix

\section{The Green's function}

For $\alpha \in ( 0, 1 )$, let $G^{ k, \delta } ( x )$ be the function that satisfies
\begin{align*}
    ( - \Delta )^{ \alpha } G^{ k, \delta } ( x ) - k^{ 2 \alpha } G^{ k, \delta } ( x ) = \delta ( x ),
\end{align*}
and admits the Fourier representation
\begin{align*}
    G^{ k, \delta } ( x )  = ( 2 \pi )^{ - d } \int_{ \mathbb{R}^{ d } } \frac{ e^{ \mathrm{i} x \cdot \xi } }{ | \xi |^{ 2 \alpha } - k^{ 2 \alpha } } d \xi.
\end{align*}
Equivalently, using spherical coordinates, it can be written as
\begin{equation*}
    G^{ k, \delta } ( x )= ( 2 \pi )^{ - \frac{ d }{ 2 } } | x |^{ - \frac{ d }{ 2 } + 1 } \int_{ 0 }^{ \infty } r^{ \frac{ d }{ 2 } } J_{ \frac{ d }{ 2 } - 1 } ( r | x | ) ( r^{ 2 \alpha } - k^{ 2 \alpha } )^{ - 1 } d r,
\end{equation*}
where $J_\nu$ is the Bessel function of order $\nu$, given explicitly by
\begin{align*}
    J_{ \nu } ( z ) = \Big( \frac{ z }{ 2 } \Big)^{ \nu } \sum_{ k = 0 }^{ \infty } ( - 1 )^{ k } \frac{ ( z^{ 2 } / 4 )^{ k } }{ k! \Gamma ( \nu + k + 1 ) }.
\end{align*}

Define the Mellin transform of $f_{ i }$ by
\begin{align*}
    g_{ i } ( s ) : = \mathcal{M} [ f_{ i } ] ( s ) = \int_{ 0 }^{ \infty } f_{ i } ( r ) r^{ s - 1 } d r, \quad i \in \{ 1, 2 \}.
\end{align*}
Let
\begin{align*}
    f_{ 1 } ( r ) = J_{ \frac{ d }{ 2 } - 1 } ( r ), \quad f_{ 2 } ( r ) = ( r^{ 2 \alpha } - k^{ 2 \alpha } )^{ - 1 },
\end{align*}
and set 
\begin{align*}
    a = 1 - \frac{ d }{ 2 }, \quad \beta = \frac{ d }{ 2 }.
\end{align*}
Using the Mellin convolution identity (see, e.g., \cite{Bateman-54}),
\begin{align*}
    \mathcal{M} \Big[ x^{ a } \int_{ 0 }^{ \infty } r^{ \beta } f_{ 1 } ( x r ) f_{ 2 } ( r ) d r \Big] = g_{ 1 } ( s + a ) g_{ 2 } ( 1 - s - a + \beta ),
\end{align*}
we obtain
\begin{align*}
    \mathcal{M} [ G^{ k, \delta } ] ( s ) &= ( 2 \pi )^{ - \frac{ d }{ 2 } } g_{ 1 } ( s + a ) g_{ 2 } ( 1 - s - a + \beta ) \\
    &= ( 2 \pi )^{ - \frac{ d }{ 2 } } g_{ 1 } ( s + 1 - \frac{ d }{ 2 } ) g_{ 2 } ( d - s ).
\end{align*}

For $\Re s \in ( 1 - \frac{ d }{ 2 }, \frac{ 3 }{ 2 } )$, we have 
\begin{align*}
    g_{ 1 } ( s ) = \int_{ 0 }^{ \infty } J_{ \frac{ d }{ 2 } - 1 } ( x ) x^{ s - 1 } d x = \frac{ 2^{ s - 1 } \Gamma ( \frac{ d }{ 4 } + \frac{ s }{ 2 } - \frac{ 1 }{ 2 } ) }{ \Gamma ( \frac{ d }{ 4 } - \frac{ s }{ 2 } + \frac{ 1 }{ 2 } ) }.
\end{align*}
Let $y = x^{ 2 \alpha }$, $\alpha \in ( 0, 1 )$. For $\Re s \in ( 0, 2 \alpha )$, 
\begin{align*}
    g_{ 2 } ( s )& = \int_{ 0 }^{ \infty } ( x^{ 2 \alpha } - k^{ 2 \alpha } )^{ - 1 } x^{ s - 1 } d x\\
    &= - \frac{ 1 }{ 2 \alpha } \int_{ 0 }^{ \infty } ( k^{ 2 \alpha } - y )^{ - 1 } y^{ \frac{ s }{ 2 \alpha } - 1 } d y = - \frac{ \pi }{ 2 \alpha } k^{ s - 2 \alpha } \cot \Big( \frac{ \pi s }{ 2 \alpha } \Big), 
\end{align*}
where the integral is interpreted in the Cauchy principal value sense. Using the reflection formulas for the Gamma function, 
\begin{align*}
    \Gamma ( z ) \Gamma ( 1 - z ) & = \frac{ \pi }{ \sin ( \pi z ) }, \quad z \neq 0, \pm 1, \cdots,
    \\ \Gamma ( z + \frac{ 1 }{ 2 } ) \Gamma ( \frac{ 1 }{ 2 } - z ) & = \frac{ \pi }{ \cos ( \pi z ) }, \quad z \neq 0, \pm 1, \cdots,
\end{align*}
we obtain, for $\Re s \in ( 0, 2 \alpha )$, 
\begin{align*}
    g_{ 2 } ( s ) = - \frac{ \pi }{ 2 \alpha } k^{ s - 2 \alpha } \frac{ \Gamma ( \frac{ s }{ 2 \alpha } ) \Gamma ( 1 - \frac{ s }{ 2 \alpha } ) }{ \Gamma ( \frac{ 1 }{ 2 } + \frac{ s }{ 2 \alpha } ) \Gamma ( \frac{ 1 }{ 2 } - \frac{ s }{ 2 \alpha } ) }.
\end{align*}
Therefore, for $\Re s \in ( 0, \frac{ d }{ 2 } + \frac{ 1 }{ 2 } ) \cap ( d - 2 \alpha, d )$, we obtain
\begin{align*}
    \mathcal{M} [ G^{ k, \delta } ] ( s ) & = ( 2 \pi )^{ - \frac{ d }{ 2 } } \frac{ 2^{ s - \frac{ d }{ 2 } } \Gamma ( \frac{ s }{ 2 } ) }{ \Gamma ( \frac{ d }{ 2 } - \frac{ s }{ 2 } ) } ( - \frac{ \pi }{ 2 \alpha } ) k^{ d - s - 2 \alpha } \frac{ \Gamma ( \frac{ d - s }{ 2 \alpha } ) \Gamma ( 1 - \frac{ d - s }{ 2 \alpha } ) }{ \Gamma ( \frac{ 1 }{ 2 } + \frac{ d - s }{ 2 \alpha } ) \Gamma ( \frac{ 1 }{ 2 } - \frac{ d - s }{ 2 \alpha } ) }
    \\ & = - 2^{ s - d - 1 } \pi^{ 1 - \frac{ d }{ 2 } } \alpha^{ - 1 } k^{ d - s - 2 \alpha } \frac{ \Gamma ( \frac{ s }{ 2 } ) \Gamma ( \frac{ s }{ 2 \alpha } + 1 - \frac{ d }{ 2 \alpha } ) \Gamma ( - \frac{ s }{ 2 \alpha } + \frac{ d }{ 2 \alpha } ) }{ \Gamma ( \frac{ s }{ 2 \alpha } - \frac{ d }{ 2 \alpha } + \frac{ 1 }{ 2 } ) \Gamma ( - \frac{ s }{ 2 } + \frac{ d }{ 2 } ) \Gamma ( - \frac{ s }{ 2 \alpha } + \frac{ d }{ 2 \alpha } + \frac{ 1 }{ 2 } ) }.
\end{align*}
The admissible strip for $s$ imposes the constraint $d < 1 + 4 \alpha$, which identifies the range of dimensions for which the Mellin representation is valid. Nevertheless, by analytic continuation, the resulting expression remains valid for all dimensions $d$.

The inverse Mellin transform can be expressed as a Mellin--Barnes contour integral. For $x\neq 0$, we have
\begin{align*}
    G^{ k, \delta } ( x )& = G^{ k, \delta } ( | x | ) = \frac{ 1 }{ 2 \pi \mathrm{ i } } \int_{ \gamma - \mathrm{ i } \infty }^{ \gamma + \mathrm{ i } \infty } \mathcal{M} [ G^{ k, \delta } ] ( s ) | x |^{ - s } d s 
    \\ & = - \frac{ 2^{ - d - 1 } \pi^{ 1 - \frac{ d }{ 2 } } \alpha^{ - 1 } k^{ d - 2 \alpha } }{ 2 \pi \mathrm{ i } } \int_{ \gamma - \mathrm{ i } \infty }^{ \gamma + \mathrm{ i } \infty } \frac{ \Gamma ( \frac{ s }{ 2 } ) \Gamma ( \frac{ s }{ 2 \alpha } + 1 - \frac{ d }{ 2 \alpha } ) \Gamma ( - \frac{ s }{ 2 \alpha } + \frac{ d }{ 2 \alpha } ) }{ \Gamma ( \frac{ s }{ 2 \alpha } - \frac{ d }{ 2 \alpha } + \frac{ 1 }{ 2 } ) \Gamma ( - \frac{ s }{ 2 } + \frac{ d }{ 2 } ) \Gamma ( - \frac{ s }{ 2 \alpha } + \frac{ d }{ 2 \alpha } + \frac{ 1 }{ 2 } ) } \Big( \frac{ k | x | }{ 2 } \Big)^{ - s } d s.
\end{align*}
The integral converges provided $\gamma \in ( \max \{ d - 2 \alpha, 0 \}, \frac{ d }{ 2 } + \frac{ 1 }{ 2 } )$. The Mellin--Barnes integral above is a Fox $H$-function. Specifically,
\begin{align}\label{149}
    & \frac{ 1 }{ 2 \pi \mathrm{ i } } \int_{ \gamma - \mathrm{ i } \infty }^{ \gamma + \mathrm{ i } \infty } \frac{ \Gamma ( \frac{ s }{ 2 } ) \Gamma ( \frac{ s }{ 2 \alpha } + 1 - \frac{ d }{ 2 \alpha } ) \Gamma ( - \frac{ s }{ 2 \alpha } + \frac{ d }{ 2 \alpha } ) }{ \Gamma ( \frac{ s }{ 2 \alpha } - \frac{ d }{ 2 \alpha } + \frac{ 1 }{ 2 } ) \Gamma ( - \frac{ s }{ 2 } + \frac{ d }{ 2 } ) \Gamma ( - \frac{ s }{ 2 \alpha } + \frac{ d }{ 2 \alpha } + \frac{ 1 }{ 2 } ) } \Big( \frac{ k | x | }{ 2 } \Big)^{ - s } d s 
    \notag \\ & = \frac{ 1 }{ 2 \pi \mathrm{ i } } \int_{ \mathcal{L} } \mathcal{H}^{ m, n }_{ p, q } ( s ) z^{ - s } d s 
    = H^{ m, n }_{ p, q } ( z ),
\end{align}
where $z=\frac{k|x|}{2}\neq 0$, and the corresponding kernel is
\begin{align*}
    \mathcal{H}^{ m, n }_{ p, q } ( s ) & = \mathcal{H}^{ 2, 1 }_{ 2, 4 } 
    \left[
    \begin{array}{c}
        ( 1 - \frac{ d }{ 2 \alpha }, \frac{ 1 }{ 2 \alpha } ), ( \frac{ 1 }{ 2 } - \frac{ d }{ 2 \alpha }, \frac{ 1 }{ 2 \alpha } )\\[5pt]
        ( 0, \frac{ 1 }{ 2 } ), ( 1 - \frac{ d }{ 2 \alpha }, \frac{ 1 }{ 2 \alpha } ), ( 1 - \frac{ d }{ 2 }, \frac{ 1 }{ 2 } ), ( \frac{ 1 }{ 2 } - \frac{ d }{ 2 \alpha }, \frac{ 1 }{ 2 \alpha } )
    \end{array}
    \Bigg| s  \right].
\end{align*}
Thus, 
\begin{align*}
 H^{ m, n }_{ p, q } ( z ) & = H^{ 2, 1 }_{ 2, 4 } 
    \left[ \frac{ k | x | }{ 2 } \Bigg|
    \begin{array}{c}
        ( 1 - \frac{ d }{ 2 \alpha }, \frac{ 1 }{ 2 \alpha } ), ( \frac{ 1 }{ 2 } - \frac{ d }{ 2 \alpha }, \frac{ 1 }{ 2 \alpha } )\\[5pt]
        ( 0, \frac{ 1 }{ 2 } ), ( 1 - \frac{ d }{ 2 \alpha }, \frac{ 1 }{ 2 \alpha } ), ( 1 - \frac{ d }{ 2 }, \frac{ 1 }{ 2 } ), ( \frac{ 1 }{ 2 } - \frac{ d }{ 2 \alpha }, \frac{ 1 }{ 2 \alpha } )
    \end{array}
    \right], 
\end{align*}
with $0 \leq m = 2 \leq q = 4$ and $0 \leq n = 1 \leq p = 2$. 

Let
\begin{align*}
    \begin{cases}
        a_{ 1 } = 1 - \frac{ d }{ 2 \alpha }, \quad a_{ 2 } = \frac{ 1 }{ 2 } - \frac{ d }{ 2 \alpha },\\[3pt]
        b_{ 1 } = 0, \quad b_{ 2 } = 1 - \frac{ d }{ 2 \alpha }, \quad b_{ 3 } = 1 - \frac{ d }{ 2 }, \quad b_{ 4 } = \frac{ 1 }{ 2 } - \frac{ d }{ 2 \alpha },\\[3pt] 
        \alpha_{ 1 } = \frac{ 1 }{ 2 \alpha },\quad \alpha_{ 2 } = \frac{ 1 }{ 2 \alpha },\\[3pt] 
        \beta_{ 1 } = \frac{ 1 }{ 2 }, \quad \beta_{ 2 } = \frac{ 1 }{ 2 \alpha }, \quad \beta_{ 3 } = \frac{ 1 }{ 2 }, \quad \beta_{ 4 } = \frac{ 1 }{ 2 \alpha }. 
    \end{cases}
\end{align*}
With these parameters, the Green's function admits the Fox $H$-function representation
\begin{align*}
    G^{ k, \delta } ( x ) = - 2^{ - d - 1 } \pi^{ 1 - \frac{ d }{ 2 } } \alpha^{ - 1 } k^{ d - 2 \alpha } H^{ 2, 1 }_{ 2, 4 } \Big( \frac{ k | x | }{ 2 } \Big).
\end{align*}
The poles of the Gamma factors in the Mellin--Barnes integrand are given by
\begin{align*}
    b_{ j l } = \frac{ - b_{ j } - l }{ \beta_{ j } }, \quad j = 1, 2, \,  l = 0, 1 , 2, \cdots, 
\end{align*}
coming from $\Gamma ( b_{ j } + \beta_{ j } s )$, and 
\begin{align*}
    a_{ i k } = \frac{ 1 - a_{ i } + k }{ \alpha_{ i } }, \quad  i = 1,\,  k = 0, 1 , 2, \cdots,
\end{align*}
coming from $\Gamma ( 1 - a_{ i } - \alpha_{ i } s )$. These two families of poles do not overlap. Indeed,
\begin{align*}
    b_{ j l } \in \{ - 2 l_{ 1 }, d - 2 \alpha - 2 \alpha l_{ 2 } \}, \quad  & l_{ 1 },\, l_{ 2 } = 0, 1, 2, \cdots,
\end{align*}
while 
\begin{equation*}
 a_{ i k } \in \{ d + 2 \alpha k \}, \quad  k = 0, 1, 2, \cdots, 
\end{equation*}
so they lie strictly to the left and right, respectively, of an appropriate vertical line. Therefore, the contour $\mathcal{L}$ in (\ref{149}) may be selected as a standard Mellin--Barnes contour separating the two sets of poles, i.e., $\gamma \in ( \max \{ d - 2 \alpha, 0 \}, \frac{ d }{ 2 } + \frac{ 1 }{ 2 })$, so that $b_{ j l } < \gamma < a_{ i k }$.

To introduce the asymptotic expansions of the Fox $H$-functions, we introduce the following parameters: 
\begin{align*}
    a^{ * } & = \sum_{ i = 1 }^{ n } \alpha_{ i } - \sum_{ i = n + 1 }^{ p } \alpha_{ i } + \sum_{ j = 1 }^{ m } \beta_{ j } - \sum_{ j = m + 1 }^{ q } \beta_{ j } = 0,
    \\ \Delta & = \sum_{ j = 1 }^{ q } \beta_{ j } - \sum_{ i = 1 }^{ p } \alpha_{ i } = 1.
\end{align*}
Since $a^{ * } = 0$ and $\Delta > 0$, by \cite[Theorem 1.9]{Kilbas-Saigo-04}, as $k | x | \to \infty$, the dominant asymptotic terms of the Fox $H$-function take the form
\begin{align*}
    H^{ 2, 1 }_{ 2, 4 } \Big( \frac{ k | x | }{ 2 } \Big) = 
    \begin{cases}
        - \frac{ \mathrm{ i } }{ \sqrt{ \pi } } e^{ \mathrm{ i } k | x | } + \frac{ \mathrm{ i } }{ \sqrt{ \pi } } e^{ - \mathrm{ i } k | x | } + O ( k^{ - 1 - 2 \alpha } | x |^{ - 1 - 2 \alpha } ), & d = 1,\\[5pt] 
       - \frac{ 1 + \mathrm{i} }{ \sqrt{ \pi k | x | } } e^{ \mathrm{ i } k | x | } + \frac{ \mathrm{i} - 1 }{ \sqrt{ \pi k | x | } } e^{ - \mathrm{ i } k | x | } + O ( k^{ - \frac{ 3 }{ 2 } } | x |^{ - \frac{ 3 }{ 2 } } ), & d = 2,\\[5pt] 
       - \frac{ 2 }{ k | x | \sqrt{ \pi } } e^{ \mathrm{ i } k | x | } - \frac{ 2 }{ k | x | \sqrt{ \pi } } e^{ - \mathrm{ i } k | x | } + O ( k^{ - 3 - 2 \alpha } | x |^{ - 3 - 2 \alpha } ), & d = 3.
    \end{cases}
\end{align*}
Consequently, as $| x | \rightarrow \infty$, the Green's function satisfies
\begin{align*}
   G^{ k, \delta } ( x ) = 
   \begin{cases}
       \frac{ \mathrm{i} k^{ 1 - 2 \alpha } }{ 4 \alpha } e^{ \mathrm{i} k | x | } - \frac{ \mathrm{i} k^{ 1 - 2 \alpha } }{ 4 \alpha } e^{ - \mathrm{ i } k | x | } + O ( k^{ - 4 \alpha } | x |^{ - 1 - 2 \alpha } ), & d = 1,\\[5pt]
       \frac{ ( 1 + \mathrm{i} ) k^{ 3 / 2 - 2 \alpha } }{ 8 \alpha \sqrt{ \pi | x | } } e^{ \mathrm{ i } k | x | } + \frac{ ( 1 - \mathrm{i} ) k^{ 3 / 2 - 2 \alpha } }{ 8 \alpha \sqrt{ \pi | x | } } e^{ - \mathrm{ i } k | x | } + O ( k^{ \frac{ 1 }{ 2 } - 2 \alpha } | x |^{ - \frac{ 3 }{ 2 } } ), & d = 2,\\[5pt] 
       \frac{ k^{ 2 - 2 \alpha } }{ 8 \alpha \pi | x | } e^{ \mathrm{ i } k | x | } + \frac{ k^{ 2 - 2 \alpha } }{ 8 \alpha \pi | x | } e^{ - \mathrm{ i } k | x | } + O ( k^{ - 4 \alpha } | x |^{ - 3 - 2 \alpha } ), & d = 3.
   \end{cases}
\end{align*}

To eliminate the contributions of incoming wave  $e^{ - \mathrm{ i } k | x | }$ in the asymptotic expansion of $G^{k, \delta}$, we introduce auxiliary functions $G^{ k, 0 }$ satisfying
\begin{align*}
    ( - \Delta )^{ \alpha } G^{ k, 0 } - k^{ 2 \alpha } G^{ k, 0 } = 0, \quad \alpha \in ( 0, 1 ).
\end{align*}
For $x \in \mathbb{R}^{ d }$, we choose
\begin{align*}
    G^{ k, 0 } ( x ) = 
    \begin{cases}
        \cos ( k | x | ), & d = 1,\\[3pt] 
        J_{ 0 } ( k | x | ), & d = 2,\\[3pt] 
        \frac{ \sin ( k | x | ) }{ | x | }, & d = 3,
    \end{cases}
\end{align*}
each of which solves the homogeneous equation.

With appropriate constants $\widetilde{C}_{ k, d, \alpha }$, the combination $G^{ k, \delta } ( x ) + \widetilde{C}_{ k, d, \alpha } G^{ k, 0 } ( x )$ eliminates the incoming term and produces a purely outgoing leading asymptotic:
\begin{align*}
    G^{ k, \delta } ( x ) + \widetilde{C}_{ k, d, \alpha } G^{ k, 0 } ( x ) = C_{ k, d, \alpha } | x |^{ \frac{ 1 - d }{ 2 } } e^{ \mathrm{ i } k | x | } + O ( | x |^{ - N_{ d } } ),
\end{align*}
where $N_d$ and $C_{ k, d, \alpha }$ are defined in \eqref{Nd} and \eqref{Ckda}, respectively, and $\widetilde{C}_{ k, d, \alpha }$ is given by 
\begin{align*}
   \widetilde{C}_{ k, d, \alpha } =
    \begin{cases}
        \frac{ \mathrm{i} k^{ 1 - 2 \alpha } }{ 2 \alpha }, & d = 1,
        \\[5pt] \frac{ \mathrm{i} k^{ 2 - 2 \alpha } }{ 4 \alpha }, & d = 2,
        \\[5pt] \frac{ \mathrm{i} k^{ 2 - 2 \alpha } }{ 4 \pi \alpha }, & d = 3.
    \end{cases}
\end{align*}

Thus we define the radiating Green's function for the fractional Helmholtz equation by
\begin{align*}
    G^{ k } ( x ) : = G^{ k, \delta } ( x ) + \widetilde{C}_{ k, d, \alpha } G^{ k, 0 } ( x ),
\end{align*}
i.e., 
\begin{align*}
    G^{ k } ( x )=
    \begin{cases}
        - \frac{ \sqrt{ \pi } k^{ 1 - 2 \alpha } }{ 4 \alpha } H^{ 2, 1 }_{ 2, 4 } ( \frac{ k | x | }{ 2 } ) + \frac{ \mathrm{i} k^{ 1 - 2 \alpha } }{ 2 \alpha } \cos ( k | x | ),  & d = 1,\\[5pt]  
        - \frac{ k^{ 2 - 2 \alpha } }{ 8 \alpha } H^{ 2, 1 }_{ 2, 4 } ( \frac{ k | x | }{ 2 } ) + \frac{ \mathrm{i} k^{ 2 - 2 \alpha } }{ 4 \alpha } J_{ 0 } ( k | x | ), & d = 2,\\[5pt] 
        - \frac{ k^{ 3 - 2 \alpha } }{ 16 \alpha \sqrt{ \pi } } H^{ 2, 1 }_{ 2, 4 } ( \frac{ k | x | }{ 2 } ) + \frac{ \mathrm{i} k^{ 2 - 2 \alpha } }{ 4 \pi \alpha } \frac{ \sin ( k | x | ) }{ | x | }, & d = 3.
   \end{cases}
\end{align*}
Using the asymptotics of the Fox $H$-function, as $| x |\to \infty$, we obtain
\begin{align}\label{Gka}
    G^{ k } ( x ) = 
    \begin{cases}
        \frac{ \mathrm{i} }{ 2 \alpha } k^{ 1 - 2\alpha } e^{ \mathrm{i} k | x | } + O ( | x |^{ - 1 - 2 \alpha } ), & d = 1,\\[5pt]
        \frac{ 1 + \mathrm{i} }{ 4 \alpha \sqrt{ \pi } } k^{ \frac{ 3 }{ 2 } - 2 \alpha } | x |^{ - \frac{ 1 }{ 2 } } e^{ \mathrm{ i } k | x | } + O ( | x |^{ - \frac{ 3 }{ 2 } } ), & d = 2,\\[5pt] 
        \frac{ 1 }{ 4 \pi \alpha } k^{ 2 - 2\alpha } | x |^{ - 1 } e^{ \mathrm{ i } k | x | } + O ( | x |^{ - 3 - 2 \alpha } ), & d = 3.
   \end{cases}
\end{align}
Finally, it is straightforward to verify that the radiating Green's function satisfies the Sommerfeld radiation condition:
\begin{align*}
    \lim_{ r \rightarrow \infty } r^{ \frac{ d - 1 }{ 2 } } ( \partial_{ r } G^{ k } - \mathrm{i} k G^{ k } ) = 0.
\end{align*}

Next, we prove the uniqueness of the Green's function.

\begin{lemma}\label{homo FHE}
Let $\alpha \in ( 0, 1 )$ and $k > 0$. If $u$ satisfies
\begin{align}\label{homogeneous FHE}
    \begin{cases}
        ( - \Delta )^{ \alpha } u - k^{ 2 \alpha } u = 0 & \text{in} \ \mathbb{ R }^{ d },
        \\[5pt] \lim \limits_{ r \rightarrow \infty } r^{ \frac{ d - 1 }{ 2 } } ( \partial_{ r } u - \mathrm{ i } k u ) = 0, & r = | x |,
    \end{cases}
\end{align}
then $u \equiv 0$.
\end{lemma}

\begin{proof}
Taking the Fourier transform of the fractional Helmholtz equation in \eqref{homogeneous FHE} yields 
\begin{align*}
    ( | \xi |^{ 2 \alpha } - k^{ 2 \alpha } ) \hat{ u } ( \xi ) = 0.
\end{align*}
Since $\xi\in\mathbb R^d$ and $k>0$, we have 
\begin{align*}
    ( | \xi |^{ 2 } - k^{ 2 } ) \hat{ u } ( \xi ) = 0. 
\end{align*}
Hence, $u$ is a solution of the Helmholtz equation along with the Sommerfeld radiation condition, i.e., 
\begin{align*}
    \begin{cases}
        - \Delta u - k^{ 2 } u = 0 & \text{in} \ \mathbb{ R }^{ d },
        \\[5pt] \lim \limits_{ r \rightarrow \infty } r^{ \frac{ d - 1 }{ 2 } } ( \partial_{ r } u - \mathrm{ i } k u ) = 0, & r = | x |,
    \end{cases}
\end{align*}
which implies $u \equiv 0$.
\end{proof}

\begin{lemma}
Let $\alpha \in ( 0, 1 )$ and $k > 0$. Then there exists a unique Green's function $G^{ k }$ satisfying
\begin{align}\label{FHE Green}
    \begin{cases}
        ( - \Delta )^{ \alpha } G^{ k } ( x ) - k^{ 2 \alpha } G^{ k } ( x ) = \delta_{ 0 } ( x ) & \text{in} \ \mathbb{ R }^{ d },
        \\[5pt] \lim \limits_{ r \rightarrow \infty } r^{ \frac{ d - 1 }{ 2 } } ( \partial_{ r } u - \mathrm{ i } k u ) = 0, & r = | x |.
    \end{cases}
\end{align}
\end{lemma}

\begin{proof}
The existence of $G^{ k }$ is proved in the appendix. Let $G^{ k, 1 }$ and $G^{ k, 2 }$ be two solutions of (\ref{FHE Green}). Define $v ( x ) : = G^{ k, 1 } ( x ) - G^{ k, 2 } ( x )$. Then $v ( x )$ satisfies (\ref{homogeneous FHE}). By Lemma \ref{homo FHE}, we obtain $v \equiv 0$. 
\end{proof}

\end{document}